\documentclass[12pt]{amsart}       
\usepackage{txfonts}
\usepackage{mathtools}	
\usepackage{amssymb}
\usepackage{eucal}
\usepackage{graphicx}
\usepackage{amsmath}
\usepackage{amscd}
\usepackage[all]{xy}           
\usepackage{amsfonts,latexsym}
\usepackage{xspace}
\usepackage{epsfig}
\usepackage{float}
\usepackage{color}
\usepackage{fancybox}
\usepackage{colordvi}
\usepackage{multicol}
\usepackage{colordvi}
\usepackage{ifpdf}
\usepackage{bm}
\usepackage{stmaryrd}
\usepackage[all]{xy}
\ifpdf
\usepackage[colorlinks,final,backref=page,hyperindex]{hyperref}
\else
\usepackage[colorlinks,final,backref=page,hyperindex,hypertex]{hyperref}
\fi
\usepackage[active]{srcltx} 

\usepackage{tikz}

\usepackage{graphicx}

\newtheorem{theorem}{Theorem}[section]
\newtheorem{prop}[theorem]{Proposition}
\newtheorem{lemma}[theorem]{Lemma}

\newtheorem{prop-def}{Proposition-Definition}[section]
\newtheorem{coro-def}{Corollary-Definition}[section]

\newtheorem{question}[theorem]{Question}

\theoremstyle{definition}
\newtheorem{defn}[theorem]{Definition}
\newtheorem{remark}[theorem]{Remark}

\newtheorem{exam}[theorem]{Example}

\newcommand{\nc}{\newcommand}
\nc{\tred}[1]{\textcolor{red}{#1}}
\nc{\tblue}[1]{\textcolor{blue}{#1}}
\nc{\tgreen}[1]{\textcolor{green}{#1}}
\nc{\tpurple}[1]{\textcolor{purple}{#1}}
\nc{\btred}[1]{\textcolor{red}{\bf #1}}
\nc{\btblue}[1]{\textcolor{blue}{\bf #1}}
\nc{\btgreen}[1]{\textcolor{green}{\bf #1}}
\nc{\btpurple}[1]{\textcolor{purple}{\bf #1}}
\nc{\frakE}{\mathfrak{E}}
\nc{\NN}{{\mathbb N}}
\nc{\ncsha}{{\mbox{\cyr X}^{\mathrm NC}}} \nc{\ncshao}{{\mbox{\cyr
			X}^{\mathrm NC}_0}}

\newcommand{\delete}[1]{}

	\nc{\mlabel}[1]{\label{#1}}
	\nc{\mcite}[1]{\cite{#1}}
	\nc{\mref}[1]{\ref{#1}}
	\nc{\meqref}[1]{\eqref{#1}}
	\nc{\mbibitem}[1]{\bibitem{#1}}

\delete{
	\nc{\mlabel}[1]{\label{#1}{\hfill \hspace{1cm}{\bf{{\ }\hfill(#1)}}}}
	\nc{\mcite}[1]{\cite{#1}{{\bf{{\ }(#1)}}}}
	\nc{\mref}[1]{\ref{#1}{{\bf{{\ }(#1)}}}}
	\nc{\meqref}[1]{\eqref{#1}{{\bf{{\ }(#1)}}}}
	\nc{\mbibitem}[1]{\bibitem[\bf #1]{#1}}
}

\nc{\sha}{{\mbox{\cyr X}}}  
\newfont{\scyr}{wncyr10 scaled 550}
\nc{\ssha}{\mbox{\bf \scyr X}}
\nc{\shap}{{\mbox{\cyrs X}}} 
\nc{\shpr}{\diamond}    
\nc{\shp}{\ast} \nc{\shplus}{\shpr^+}
\nc{\shprc}{\shpr_c}    
\nc{\dep}{\mrm{dep}} \nc{\lc}{\lfloor} \nc{\rc}{\rfloor}
\nc{\db}{\leq_{\rm db}} \nc{\bfk}{\bf k}

\nc{\cala}{{\mathcal A}} \nc{\calb}{{\mathcal B}}
\nc{\calc}{{\mathcal C}}
\nc{\cald}{{\mathcal D}} \nc{\cale}{{\mathcal E}}
\nc{\calf}{{\mathcal F}} \nc{\calg}{{\mathcal G}}
\nc{\calh}{{\mathcal H}} \nc{\cali}{{\mathcal I}}
\nc{\call}{{\mathcal L}} \nc{\calm}{{\mathcal M}}
\nc{\caln}{{\mathcal N}} \nc{\calo}{{\mathcal O}}
\nc{\calp}{{\mathcal P}} \nc{\calr}{{\mathcal R}}
\nc{\cals}{{\mathcal S}} \nc{\calt}{{\mathcal T}}
\nc{\calu}{{\mathcal U}} \nc{\calw}{{\mathcal W}} \nc{\calk}{{\mathcal K}}
\nc{\calx}{{\mathcal X}} \nc{\CA}{\mathcal{A}}
\nc{\LK}{\mathfrak{L}}
\nc{\RE}{\Gamma}
\nc{\re}{\gamma}

\nc{\fraka}{{\mathfrak a}} \nc{\frakA}{{\mathfrak A}}
\nc{\frakb}{{\mathfrak b}} \nc{\frakB}{{\mathfrak B}}
\nc{\frakc}{{\mathfrak c}}
\nc{\frakD}{{\mathfrak D}} \nc{\frakF}{\mathfrak{F}}
\nc{\frakf}{{\mathfrak f}} \nc{\frakg}{{\mathfrak g}}
\nc{\frakH}{{\mathfrak H}} \nc{\frakL}{{\mathfrak L}}
\nc{\frakM}{{\mathfrak M}} \nc{\bfrakM}{\overline{\frakM}}
\nc{\frakm}{{\mathfrak m}} \nc{\frakP}{{\mathfrak P}}
\nc{\frakN}{{\mathfrak N}} \nc{\frakp}{{\mathfrak p}}
\nc{\frakS}{{\mathfrak S}} \nc{\frakT}{\mathfrak{T}}
\nc{\frakX}{{\mathfrak X}}
\nc{\frakZ}{\mathfrak{Z}}
\nc{\frakJ}{\mathfrak{J}}
\nc{\mfrakL}{\Phi}
\nc{\mfrakH}{\Psi}
\nc{\MfrakL}{\phi}
\nc{\MfrakH}{\psi}
\nc{\frakR}{\mathfrak{R}}
\nc{\GL}{\mathrm{GL}}
\nc{\gl}{\mathfrak{gl}}
\nc{\frakh}{\mathfrak{h}}
\nc{\xsj}{\vartriangleright}

\font\cyr=wncyr10 \font\cyrs=wncyr7
\nc{\li}[1]{\textcolor{red}{#1}}
\nc{\lir}[1]{\textcolor{red}{Li:#1}}
\nc{\zong}[1]{\textcolor{blue}{Zong: #1}}
\nc{\xing}[1]{\textcolor{blue}{Xing:#1}}
\nc{\revise}[1]{\textcolor{red}{#1}}
\nc{\yi}[1]{\textcolor{green}{Yi:#1}}

\nc{\dd}{{\rm d}} \nc{\Ad}{{\rm AD}}
\nc{\CC}{\mathbb{C}} \nc{\PP}{\mathbb{P}}
\nc{\QQ}{\mathbb{Q}} \nc{\ZZ}{\mathbb{Z}}
\nc{\ZZZ}{\mathbb{Z}^\ast} \nc{\RR}{\mathbb{R}} \nc{\id}{{\rm id}}
\nc{\AD}{{\rm AD}} \nc{\aad}{{\rm Ad}} \nc{\pown}{P_n} \nc{\powm}{P_m}
\nc{\Aut}{{\rm Aut}}
\nc{\Der}{{\rm Der}}
\nc{\ad}{{\rm Ad}} \nc{\fni}{\frac{1}{n}} \nc{\fmi}{\frac{1}{m}}
\nc{\ada}{\rm ad} \nc{\mulz}{\cdot_0}

\nc{\complim}{synchronized\xspace}
\nc{\compgroup}{synchronized\xspace}

\nc{\limwtzero}{limit-weight zero\xspace}

\nc{\compint}{synchronized integrable\xspace}
\nc{\no}{\precsim}
\nc{\No}{\prec}
\nc{\newc}[1]{{\bf #1}}

\begin{document}

\title[Rota-Baxter operators, differential operator, pre- and Novikov structures]{Rota-Baxter operators, differential operators, pre- and Novikov structures on groups and Lie algebras
}
%
\author{Xing Gao}
\address{School of Mathematics and Statistics, Lanzhou University
	Lanzhou, 730000, China;
	Gansu Provincial Research Center for Basic Disciplines of Mathematics
	and Statistics
	Lanzhou, 730070, China
}
\email{gaoxing@lzu.edu.cn}

\author{Li Guo}
\address{
	Department of Mathematics and Computer Science,
	Rutgers University,
	Newark, NJ 07102, United States}
\email{liguo@rutgers.edu}

\author{Zongjian Han}
\address{School of Mathematical Sciences, Tonji University, Shanghai, 200092, P.\,R. China}
\email{dblnhzj@163.com}

\author{Yi Zhang}
\address{School of Mathematics and Statistics; Center for Applied Mathematics of Jiangsu Province/Jiangsu
International Joint Laboratory on System Modeling and Data Analysis, NUIST, Nanjing, Jiangsu, 210044, P.\,R. China}
\email{zhangy2016@nuist.edu.cn}

\date{\today}
\begin{abstract}
Rota-Baxter operators on various structures have found important applications in diverse areas, from renormalization of quantum field theory to Yang-Baxter equations. Relative Rota-Baxter operators on Lie algebras are closely related to pre-Lie algebras and post-Lie algebras. Some of their group counterparts have been introduced to study post-groups, skew left braces and set-theoretic solutions of Yang-Baxter equations, but searching suitable notions of relative Rota-Baxter operators on groups with weight zero and pre-groups has been challenging and has been the focus of recent studies, by provisionally imposing an abelian condition.

Arising from the works of Balinsky-Novikov and Gelfand-Dorfman, Novikov algebras and their constructions from differential commutative algebras have led to broad applications. Finding their suitable counterparts for groups and Lie algebras has also attracted quite much recent attention.

This paper uses one-sided-inverse pairs of maps to give a perturbative approach to a general notion of relative Rota-Baxter operators and differential operators on a group and a Lie algebra with limit-weight. With the extra condition of limit-abelianess on the group or Lie algebra, we give an interpretation of relative Rota-Baxter and differential operators with weight zero. These operators motivate us to define pre-groups and Novikov groups respectively as the induced structures. The tangent maps of these operators are shown to give Rota-Baxter and differential operators with weight zero on Lie algebras. The tangent spaces of the pre-Lie and Novikov Lie groups are pre-Lie algebras and Novikov Lie algebras, fulfilling the expected property. Furthermore, limit-weight relative Rota-Baxter operators on groups give rise to skew left braces and then set-theoretic solutions of the Yang-Baxter equation.
\end{abstract}

\makeatletter
\@namedef{subjclassname@2020}{\textup{2020} Mathematics Subject Classification}
\makeatother
\subjclass[2020]{
22E60, 
17B38, 
17B40, 
16W99, 
45N05 
}

\keywords{Rota-Baxter group; differential group; pre-group; Novikov group; Novikov Lie algebra; brace; Yang-Baxter equation}

\maketitle

\vspace{-1cm}
\tableofcontents

\setcounter{section}{0}

\allowdisplaybreaks

\section{Introduction}
This paper uses a sequence limit of one-sided-inverse pairs of maps as a perturbation device to generalize the notions of relative Rota-Baxter operators and relative differential operators on groups from the existing case of weight $\pm 1$ to the case of limit-weights. A similar device defines a limit-abelianess for the groups on which the limit-weighted operators induce a notion of pre-groups and Novikov groups, which are sought after recently. Relations to relative Rota-Baxter and differential Lie algebras, skew left braces and the quantum Yang-Baxter equation are also discussed.

In this introduction, we begin with the background and key questions for the Rota-Baxter operators in Section~\mref{ss:rbback} and the differential operators in Section~\mref{ss:diffback}. We then sketch our approach to address these questions in Section~\mref{ss:laa} and give a summary of the main results in Section~\mref{ss:sum}

\subsection{Rota-Baxter operators and pre-structures on Lie algebras and groups} \mlabel{ss:rbback}
We first give the motivation and background on Rota-Baxter operators on Lie algebras, their related pre-Lie algebras and the classical Yang-Baxter equation, followed by recent progresses for these structures on groups, leading to the questions to be addressed.

\subsubsection{Rota-Baxter Lie algebras, classical Yang-Baxter equation and pre-Lie algebras}
For a vector space $R$ equipped with a binary operation $*$ and a fixed scalar $\lambda$, a linear operator $P:R\to R$ is called a {\bf Rota-Baxter operator on $(R,*)$ with weight $\lambda$} if
\begin{equation*}
	P(x)*P(y)=P\big(x*P(y)\big)+P\big(P(x)*y\big)+\lambda P(x*y), \ \quad x,y \in R.
	\mlabel{eq:rb}
\end{equation*}
When $*$ is the associative product, the notion originated from the 1960 work of G.~Baxter~\mcite{Ba}, and its affine transformation $Q:=-\lambda \id -2P$ can be traced even further to the 1951 work of Tricomi~\mcite{Tri}. It has attracted a lot of attention in recent years with broad applications~(see for example~\mcite{CK,Gub,Ro}).

For the Lie algebra, both the Rota-Baxter operator and its affine transformation were independently discovered by Semenov-Tian-Shansky in~\mcite{STS} as the {\bf operator form} of the important {\bf classical Yang-Baxter equation (CYBE)}~\mcite{BD}, first given in the tensor form
\begin{equation}\notag
	[r_{12},r_{13}]+[r_{12},r_{23}]+[r_{13},r_{23}]=0,
	\mlabel{eq:cybe}
\end{equation}
where $r\in \frakg\otimes \frakg$ and $\frakg$ is a Lie algebra.
The equation arose from the study of
inverse scattering theory in the 1980s and then was recognized as the
``semi-classical limit" of the quantum Yang-Baxter equation following the works of C.~N. Yang~\mcite{Ya} and R.~J. Baxter~\mcite{BaR}.
CYBE is further related to classical integrable
systems and quantum groups~\mcite{CP}. The operator forms of the CYBE serve the purposes of both understanding the CYBE and of generalizing the equation. In fact, the affine transformation was called the {\bf modified Yang-Baxter equation} which does not have a tensor form like the CYBE.

Further generalizing the operator approach to the CYBE, Kupershmidt~\mcite{Ku} introduced the notion of {\bf $O$-operators}, later often called {\bf relative Rota-Baxter operators}. Every antisymmetric solution of the CYBE gives a relative Rota-Baxter operator with weight zero. In the other direction, the work~\mcite{Bai2} of Bai provided an inverse procedure that produces an antisymmetric solution of the CYBE from every relative Rota-Baxter operator with weight zero.
Furthermore, a relative Rota-Baxter operator with weight zero or one is closely related to a pre-Lie or post-Lie algebra respectively.
These relations can be summarized in the following diagram.
\vspace{-.1cm}
\begin{equation}
	\begin{split}
		\xymatrix{
			\txt{pre-Lie (resp. post-Lie) \\ algebras} \ar@<.4ex>[r]      & \txt{relative Rota-Baxter operators with \\ weight zero (resp. one) on Lie algebras} \ar@<.4ex>[l]
			\ar@<.4ex>[r]&
			\txt{solutions of \\ CYBE} \ar@<.4ex>[l]
		}
	\end{split}
	\mlabel{eq:bigdiag}
\end{equation}
Here the arrows for CYBE are for operators with weight zero.

\subsubsection{Rota-Baxter groups, braces, the quantum Yang-Baxter equation and post-groups}

Understanding the notions and relations in Diagram~\meqref{eq:bigdiag} on the group level has been the motivation of long-time studies.
Back in the 1980s, as a fundamental lemma to applications in integrable systems~\mcite{FRS,RS1,RS2}, Semenov-Tian-Shansky obtained a
Global Factorization Theorem for a Lie group from integrating an Infinitesimal Factorization Theorem for a Lie algebra, making use of the modified Yang-Baxter equation (equivalently, a Rota-Baxter operator with weight one).

To obtain the Global Factorization Theorem directly on the Lie group level, the notion of Rota-Baxter operators on groups with weight one was introduced in~\mcite{GLS}, leading to many developments. The tangent maps of these operators are Rota-Baxter operators on the Lie algebras with weight one.
More generally, a relative Rota-Baxter operator with weight one was defined in~\mcite{JSZ}.
In the context of Diagram~\meqref{eq:bigdiag}, on the one hand, Rota-Baxter operators on groups give skew left braces~\mcite{BG} which give set-theoretic solutions of the quantum Yang-Baxter equation~\mcite{GV}; on the other hand, the operators are also closely related to the newly introduced notion of post-groups~\mcite{BGST}. These notions and connections are shown in the following diagram.
\vspace{-.2cm}
\begin{equation}
	\mlabel{eq:bigdiag2}
	\begin{split}
		\xymatrix{
			\txt{post-Lie \\ groups} \ar@<.4ex>[r]      & \txt{relative Rota-Baxter operators \\ with weight one on Lie groups}\ar@<.4ex>[l] \ar@<.4ex>[r]&
			\txt{solutions of \\ YBE}
		}
	\end{split}
\vspace{-.5cm}
\end{equation}
\vspace{-.9cm}
\subsubsection{In search of Rota-Baxter groups with weight zero and pre-groups}
An apparent discrepancy can be noticed at this point between the diagrams in~\meqref{eq:bigdiag} and \meqref{eq:bigdiag2}: the relative Rota-Baxter operators on Lie algebras in Diagram~\meqref{eq:bigdiag} have weights zero and one while those operators on groups in the above-mentioned recent progresses are only defined for weight one. Correspondingly, the group counterpart of pre-Lie algebras is missing in Diagram~\meqref{eq:bigdiag}.
This situation can be summarized in the following question.
\vspace{-.2cm}
\begin{question}
\mlabel{qu:rbopr}
Find suitable notions of a Rota-Baxter operator on a group with weight zero and of a pre-Lie group, such that the following expected relations hold.
\vspace{-.2cm}
\begin{equation*}
\mlabel{eq:rbpre}
\begin{split}		
\xymatrix{
\text{{\small \newc{pre-Lie groups} }}  \ar[d]_{\rm tangent}^{\rm space} &&& \ar[d]_{\rm tangent}^{\rm space}  \text{{\small \newc{weight zero Rota-Baxter groups} }}
\ar[lll] \\
			 \text{{\rm {\small  pre-Lie  algebras} } } &&&
			 \text{{\rm {\small   weight zero Rota-Baxter Lie algebras }} }
\ar[lll]
}
\end{split}
\end{equation*}
\vspace{-.4cm}
\end{question}
\vspace{-.2cm}
Addressing this question has been the focus of several recent studies~\mcite{BGST,BN,GGLZ,LST1,LST2}. Motivated by the fact that a Rota-Baxter operator with weight one on an abelian Lie algebra is automatically a Rota-Baxter operator with weight zero, an abelian condition on a group is usually imposed as the definition of a Rota-Baxter operator with weight zero on a group. Likewise, a post-group is called a pre-group if the underlying Lie group is abelian.
Improving this restricted condition is hindered by the rigidity of the group structure compared to algebras with linear structures such as the Lie algebra.

The first goal of this paper is to address this question based on the more general notion of limit-abelianess of groups and Lie algebras, sketched in Section~\mref{ss:laa}.
\vspace{-.3cm}
\subsection{Differential operators and Novikov structures on Lie algebras and groups} \mlabel{ss:diffback}
The understanding on differential operators on Lie algebras and groups, and their induced structures is even more limited.
\vspace{-.2cm}
\subsubsection{Differential commutative algebras and Novikov algebras}
Having the Rota-Baxter operator with weight $\lambda$ as the formal inverse, a {\bf differential operator with weight $\lambda$} on a vector space $R$ with a binary operation $\ast$ is a linear operator $d:R\to R$ satisfying
\begin{equation*}
	d(x*y)=d(x)*y+x*d(y)+\lambda d(x)*d(y), \quad x, y\in R.
\mlabel{eq:difop}
\end{equation*}
The operator is also called a {\bf derivation} when the weight is zero, in which case its study over a field is the differential algebra initiated by Ritt and Kolchin as an algebraic study of differential equations and has expanded into a vast area of research and applications~\mcite{Ko,PS,Ri}.

In applications, a classical construction of Gelfand states that a derivation $d$ on a commutative associative algebra $R$ gives rise to a (left) Novikov algebra by $a\circ b:=ad(b)$, similar to the fact that a Rota-Baxter operator with weight zero on a Lie algebra defines a pre-Lie algebra.

With its own significance, Novikov algebras emerged from S. Novikov's work in the early 1980s~\mcite{BN85} on Poisson brackets of hydrodynamic type. Their development was closely linked to integrable systems, Poisson geometry, and the broader mathematical framework established by researchers like Gelfand and Dorfman~\mcite{GD}. Recently, the Novikov algebra and its related algebraic structures continue to be a rich area of study, with ongoing research exploring their theoretical properties and practical applications in mathematics and physics~\mcite{GLB,HBG,KMS,KSO,SK}. Also Novikov algebras are a subclass of pre-Lie algebras with its own broad applications.

Beyond Novikov algebras and their construction from derivations on commutative algebras, a notion of noncommutative Novikov algebra was also defined by Sartayev and Kolesnikov~\mcite{SK} and shown to be induced from a derivation on an associative algebra.

\vspace{-.2cm}
\subsubsection{In search of derivations on groups, Novikov Lie algebras and Novikov groups}
It is known that a derivation on a Lie algebra induces a magmatic algebra without the original Lie bracket~\mcite{KSO}. It would be interesting to obtain a Lie version of the Novikov algebra retaining the original Lie bracket.
On the group level, it is not known how to define a derivation on a group or a group variation of the Novikov algebra, even if a differential operator on a group with weight one has its relative version as the crossed homomorphism on a group~\mcite{GLS}.

The overall goal to understand this situation can be summarized in the following question (the boldfaced notions are to be defined).
\vspace{-.3cm}
\begin{question}
\mlabel{qu:diffnov}
Find suitable notions of Novikov Lie algebras and Novikov Lie groups, as well as derivations on Lie groups, such that the expected relations in the following diagram hold.
\vspace{-.3cm}
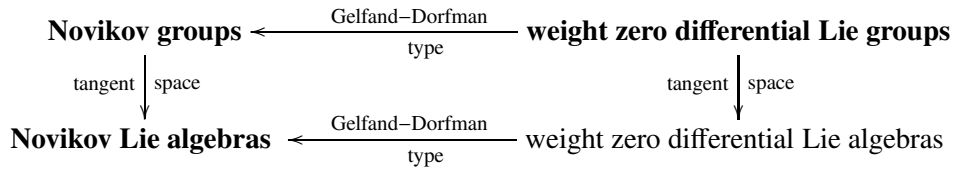
\begin{equation*}
\mlabel{eq:diffnov}
\begin{split}
\xymatrix{
 \text{{\small \newc{Novikov groups}}}\ar[d]_{\rm tangent}^{\rm space}&&& \text{{\small \newc{weight zero differential Lie groups}}} \ar[lll]_{\rm Gelfand-Dorfman \quad \quad}^{\rm type\quad}\ar[d]_{\rm tangent}^{\rm space} \\
\text{{\small \newc{Novikov Lie algebras}} } &&&
\text{{\small {\rm weight zero differential Lie algebras} }} \ar[lll]_{\rm Gelfand-Dorfman\quad\quad}^{\rm type\quad}
}
\end{split}
\end{equation*}
\vspace{-.4cm}
\end{question}

The second goal of this paper is to address this question based on the notion of limit-abelianess of groups and Lie algebras.
\vspace{-.2cm}
\subsection{The limit-abelian approach to the two problems}
\mlabel{ss:laa}

In a recent paper~\mcite{GGH}, we use the notion of Rota-Baxter operators on groups and Lie algebras with limit-weights to give a new notion of Rota-Baxter operators with weight zero on groups, by a type of deformations of the weight one case. Then the notion is applied to study group-values integrals.

In the present paper, we first generalize these Rota-Baxter operators with limit-weights to the relative context. We then use a one-sided-inverse pair of maps on a group or Lie algebra to define a weak transported structure. From the sequence of such map pairs in a limit-weighted group or a limit-weighted Lie algebra, a limit process of structure transportations of the original multiplication can be obtained. The abelianess of the limit multiplication is called the {\bf limit-abelian property}, and is used to define a pre-Lie group from a limit-weighted post-Lie group. A limit-weighted relative Rota-Baxter operator on a limit-abelian group (resp. Lie algebra) gives a pre-group (resp. pre-Lie algebra).

The differential case is taken care of by a similar approach, except that here there is no counterpart for the post-Lie algebra or the post-group, reflected by the two-dimensional diagram in the differential case in Figure~\mref{fig:diffalggp} (p.~ \pageref{fig:diffalggp}) verses the three-dimensional diagram in the Rota-Baxter case in Figure~\mref{fig:rbalggp} (p.~\pageref{fig:rbalggp}). Also, with our approach, the notion of a Novikov Lie algebra with limit-weight is naturally defined in analog to a post-Lie algebra with limit-weight and then gives a notion of a Novikov Lie algebra (without limit) by adding the original Lie algebra bracket into consideration.

\subsection{The outline} \mlabel{ss:sum}
We now give an outline of the paper with the goals of addressing Questions~\mref{qu:rbopr} and \mref{qu:diffnov}. For Question~\mref{qu:rbopr}, the notions of Rota-Baxter operators on Lie groups and Lie algebras, pre-Lie groups and their relation are depicted in Figure~\mref{fig:rbalggp} in page~\pageref{fig:rbalggp}, where Rota-Baxter is abbreviated as RB. For Question~\mref{qu:diffnov}, the notions of derivations on Lie groups and on Lie algebras, and Novikov Lie groups and Novikov Lie algebras, together with their relations are summarized in Figure~\ref{fig:diffalggp} in page~\pageref{fig:diffalggp}.

\begin{figure}
	\caption{Rota-Baxter operators and pre-structures }
	\mlabel{fig:rbalggp}
{\small
\begin{displaymath}
	\xymatrix@R=0.8cm@C=0.55cm{
		&\txt{\small  \newc{limit-weighted}\\ \small \newc{post-Lie groups}} \ar@{<-}[rr]_{\tiny \rm Proposition~\ref{prop:preRB}}\ar@{->}[ld]|-{\txt{\tiny limit-abelian\\ \tiny Definition~\ref{defn:pregroup}}}\hole \ar@{.>}[dd]|-(0.8){\txt{\tiny tangent space\\ \tiny Theorem~\ref{thm:tangprelie}}}\hole& &\txt{\small \newc{limit-weighted} \\ \small  \newc{relative RB operators} \\ \small \newc{on Lie groups} }\ar@{->}[ld]|-{\txt{\tiny limit-abelian\\ \tiny Definition~\ref{defn:pregroup}}}\hole\ar@{->}[dd]|-{\txt{\tiny tangent space\\ \tiny Theorem~\ref{thm:tgrelim}}}\hole&\\
		\txt{\small \newc{pre-Lie groups}}\ar@{->}[dd]|-{\txt{\tiny tangent space\\ \tiny Theorem~\ref{thm:tangprelie}}}\hole&&\txt{\small \newc{limit-weighted}\\
			\small \newc{relative RB operators on}\\ \small \newc{limit-abelian Lie groups}}\ar@{->}[dd]|-(0.2){\txt{\tiny tangent space\\ \tiny Remark~\ref{rk:tgab}}}\hole \ar@{->}[ll]^(0.37){\tiny\rm Proposition~\ref{prop:preRB}\quad \quad }\\
		&\txt{\small \newc{limit-weighted}\\ \small \newc{post-Lie algebras}}\ar@{.>}[dd]|-(0.75){ \txt{\tiny $\MfrakL_{\frac{1}{n}}=\MfrakH_{\frac{1}{n}}:= \id$ \\ \tiny Remark~\ref{rk:post-pre}}} & &\txt{\small \newc{limit-weighted} \\ \small \newc{relative RB operators} \\ \small \newc{on Lie algebras} }\ar@{.>}[ll]^(0.7){\rm Proposition~\ref{prop:rbpre}}\ar@{->}[dd]|-{\txt{\tiny $\MfrakL_{\frac{1}{n}}=\MfrakH_{\frac{1}{n}}:= \id$ \\ \tiny Example~\ref{rk:relrblie}~\eqref{rk:re weight lambda}}}\hole\\
		\txt{\small pre-Lie algebras}\ar@{=}[dd]\ar@{<.}[ru]|-{\txt{\tiny limit-abelian \\ \tiny Remark~\ref{rk:post-pre}}}\hole &&\txt{\small \newc{limit-weighted}\\
			\small \newc{relative RB operators on}\\ \small \newc{limit-abelian Lie algebras}}\ar@{->}[ll]^(0.39){\tiny \rm Proposition~\ref{prop:rbpre}}\ar@{<-}[ru]\ar@{->}[dd]|-(0.2){ \txt{\tiny $\MfrakL_{\frac{1}{n}}:=\frac{1}{n} \id$ , $\MfrakH_{\frac{1}{n}}:=n\, \id$ \\ \tiny Example~\ref{rk:relrblie}~\eqref{rk:re weight lambda}}}\hole\\
		&\txt{\small post-Lie algebras}\ar@{<.}[rr]\ar@{.>}[ld]|{\txt{\tiny abelian } }\hole&&\txt{\small weight one  \\ \small relative RB operators \\ \small on Lie algebras}\\
		\txt{ \small pre-Lie algebras}\ar@{<-}[rr]&&\txt{\small weight zero \\ \small relative RB operators\\ \small on Lie algebras}\ar@{<-}[ru]|-{\txt{\tiny weighted term vanish}}\hole&
	}
\end{displaymath}
}
\noindent
{\small Note: In the diagram, the boldfaced terms are newly introduced. The classical results are summarized at the bottom level. The generalizations for Lie algebras are at the middle level, and the lifted results for Lie groups are at the top level.}
\end{figure}
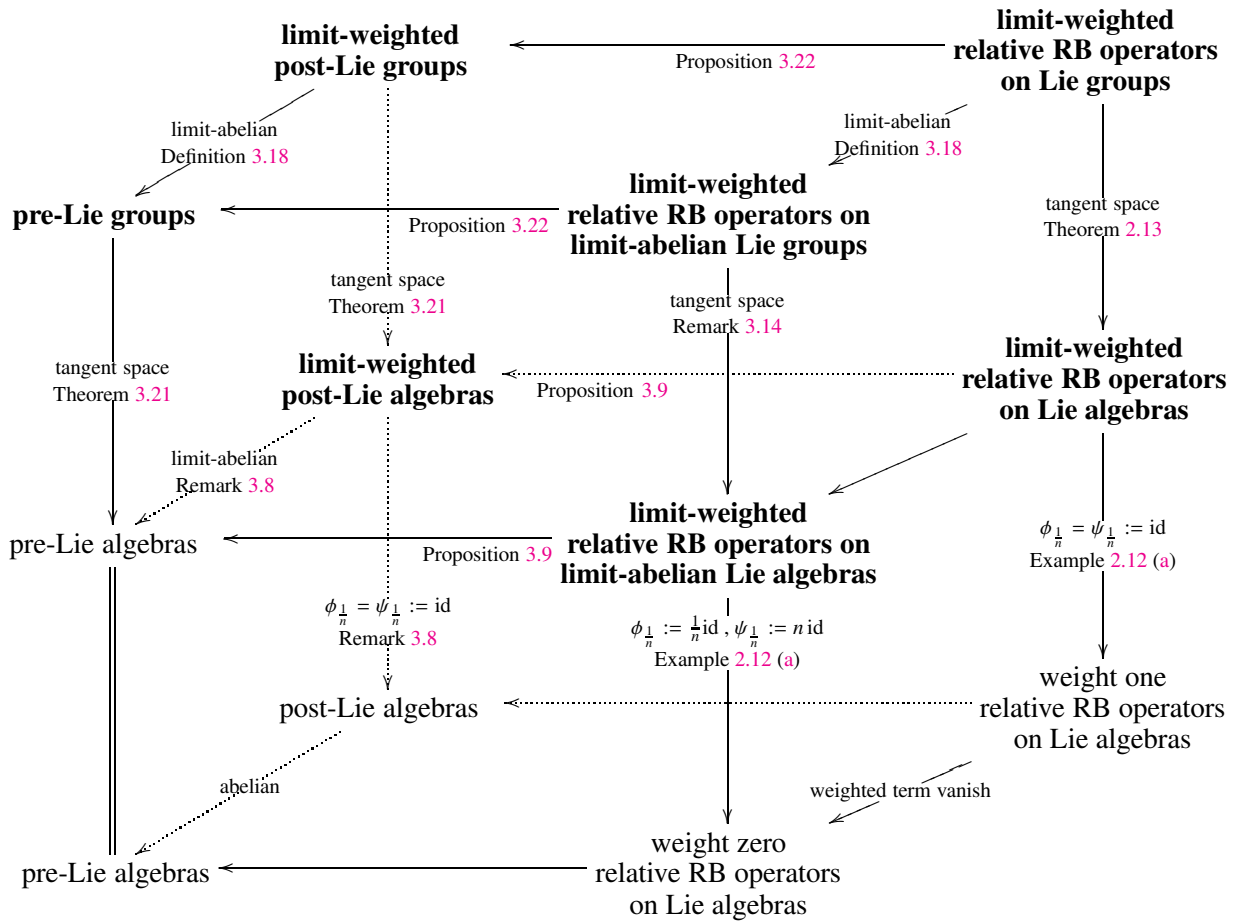

\subsubsection{Relative Rota-Baxter groups with limit-weights}

In Section~\mref{sec:pairlimit}, after reviewing notions on \complim groups and Lie algebras, we define relative Rota-Baxter operators with limit-weights on topological groups and Lie algebras, with the Lie algebras serving as the tangent spaces of the Lie groups. Their descent groups are also obtained. We then define an integration for functions with group values, as a relative variation of the integration defined in~\mcite{GGH}.

\subsubsection{Limit-abelian relative Rota-Baxter groups with limit-weight and pre-Lie groups}
In Section~\mref{sec:ybe} we define a new notion of pre-groups, to be post-groups satisfying certain limit-abelian condition (Definition~\mref{defn:pregroup}). It fulfills the expected properties that its tangent map gives a pre-Lie algebra (Theorem~\mref{thm:tangprelie}), and that it is derived from a limit-weighted relative Rota-Baxter operator on a limit-abelian group (Proposition~\mref{prop:preRB}). We introduce post-semigroups and use them to define limit-weighted post-groups. Then we define pre-groups based on the limit-abelian condition. Without the limit-abelian condition, a relative Rota-Baxter operator on a Lie algebra of limit-weight gives rise to a generalization of the post-Lie algebra to the limit-weight case (Proposition~\mref{prop:rbpre}) which also yields a pre-Lie algebra under the limit-abelian condition (Remark~\mref{rk:post-pre}). Thus relative Rota-Baxter operators with limit-weights provide a general framework not only for post-groups and pre-groups, but also for their classical linear counterparts, the post-Lie algebras.

We finally show that limit-weighted Rota-Baxter groups provide a rich source to produce skew left braces and hence set-theoretic solutions of the Yang-Baxter equation (Proposition~\mref{pp:slotoybe2}).

\subsubsection{Relative differential operators on groups, Novikov groups and Novikov Lie algebras}

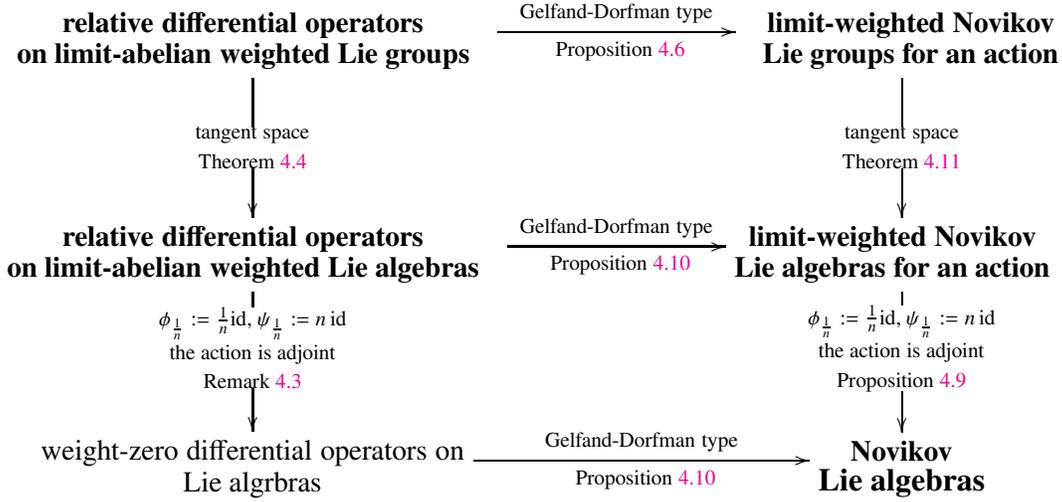
\begin{figure}
\caption{Differential operators and Novikov structures}
\mlabel{fig:diffalggp}
\begin{displaymath}
\xymatrix@R=0.8cm@C=0.8cm{
			\txt{\small \newc{relative differential operators} \\ \small \newc{on limit-abelian weighted Lie groups}} \ar[rrr]^-(0.5){\txt{\tiny Gelfand-Dorfman type}}_-(0.5){\txt{ \tiny Proposition~\ref{pp:novidf}}} \ar@{->}[dd]|-(0.5){\txt{\tiny tangent space\\ \tiny Theorem~\ref{thm:dgpdla}}}\hole&&& \ar@{->}[dd]|-(0.5){\txt{\tiny tangent space\\ \tiny Theorem~\ref{thm:ngpnla}}}\hole\txt{\small \newc{limit-weighted Novikov} \\ \small \newc{Lie groups for an action} } \\
			&&& \\
			\txt{\small \newc{relative differential operators} \\ \small \newc{on limit-abelian weighted Lie algebras}} \ar@{->}[dd]|-(0.4){ \txt{\tiny $\MfrakL_{\frac{1}{n}}:=\frac{1}{n} \id$, $\MfrakH_{\frac{1}{n}}:=n\, \id$ \\ \tiny  the action is adjoint \\ \tiny Remark~\ref{rk:redf}}}\hole \ar[rrr]^-(0.5){\txt{\tiny Gelfand-Dorfman type}}_-(0.5){\txt{ \tiny Proposition~\ref{pp:dernov}}} &&&
			\txt{\small \newc{limit-weighted Novikov} \\ \small \newc{Lie algebras for an action}} \ar@{->}[dd]|-(0.4){ \txt{\tiny $\MfrakL_{\frac{1}{n}}:=\frac{1}{n} \id$, $\MfrakH_{\frac{1}{n}}:=n\, \id$ \\ \tiny  the action is adjoint \\ \tiny Proposition~\ref{pp:dervex}}}\hole
			\\
			&&& \\
			\txt{\small weight-zero differential operators on \\ \small Lie algrbras} \ar[rrr]^-(0.5){\txt{\tiny Gelfand-Dorfman type}}_-(0.5){\txt{ \tiny Proposition~\ref{pp:dernov}}} &&&
			\txt{\small \newc{Novikov} \\ \newc{Lie algebras}}
		}
	\end{displaymath}
	Note: The boldfaced terms are newly introduced.
\end{figure}

Section~\mref{sec:nov} addresses Question~\mref{qu:diffnov} with the new notions and relations summarized in Figure~\ref{fig:diffalggp} in page~\pageref{fig:diffalggp}.
We first introduce relative differential operators on groups and Lie algebras with limit-weights, with tangent maps of the former giving the latter (Theorem~\mref{thm:dgpdla}). We then give the notion of Novikov groups with limit-weight and in particular Novikov groups for $\QQ$-groups. The Galfand-Dorfman construction of Novikov algebras from derivations on commutative algebras has its group analog that a relative differential operator on a Lie group with limit-weight gives rise to a Novikov group with limit-weight. Conversely, a Novikov group with limit-weight also gives a relative differential operator with limit-weight (Proposition~\mref{pp:novidf}). Finally, the notions of a Novikov Lie algebra with limit-weight and in particular a Novikov Lie algebra are introduced, which serves as the tangent space of a Novikov Lie group with limit-weight on the one hand (Theorem~\mref{thm:ngpnla}), and as the induced structure of a relative differential operator on a Lie algebra on the other hand (Proposition~\mref{pp:dernov}).

\smallskip

\noindent
{\bf Notations.} Throughout this paper, we will work over a fixed field $\bfk$ of characteristic zero. It is the base field for all vector spaces, algebras,  tensor products, as well as linear maps.
Denote by $\mathbb{P}$ the set of positive ones. Let $e$ be the identity of a group.

\vspace{-.3cm}

\section{Limit-weighted relative Rota-Baxter groups and Lie algebras}
\mlabel{sec:pairlimit}
In this section, we introduce the notion of a relative Rota-Baxter operator on a group with weight given by an infinite sequence of map pairs. This notion generalizes the one in~\mcite{GGH} to the relative context.

\subsection{Action-\complim groups and Lie algebras}
\mlabel{ss:actionsynch}

We first recall the following basic notions on map pairs.

\begin{defn}~\cite{GGH}
\mlabel{defn:pair}
Let $G$ be a group and $\mfrakL, \mfrakH:G\rightarrow G$ a pair of maps.
\begin{enumerate}
\item     We call $(G,(\mfrakL,\mfrakH))$, or simply $G$,  a {\bf $(\mfrakL,\mfrakH)$-group} if $\mfrakL \mfrakH= \id_G$.
\item We call the pair $(\mfrakL,\mfrakH)$ {\bf bijective} if $\mfrakL$ is bijective with $\mfrakL^{-1}=\mfrakH$.
\item We call the pair $(\mfrakL,\mfrakH)$ {\bf unital} if $\mfrakL(e)=\mfrakH(e)=e$.
\item
 Let $\frakg$ be a Lie algebra and $\MfrakL,\MfrakH, B:\frakg \rightarrow \frakg$ be linear maps.
We call $(\frakg,(\MfrakL,\MfrakH))$ a {\bf $(\MfrakL,\MfrakH)$-Lie algebra} if $\MfrakL \MfrakH= \id_\frakg$.
\end{enumerate}
\end{defn}

We will also need the following condition for an iterated limit to be obtained by taking the limit in the diagonal direction.

\begin{defn}~\cite{GGH}
\mlabel{defn:synchronized-limit}
Let $T$ be a topological space and $f_n: T\times T\rightarrow T, n\geq 1,$ be a sequence of maps such that $\lim\limits_{n\to \infty} f_n(a,b)$ uniquely exists for each $(a,b)\in T\times T$. We call the resulting limit function $\lim\limits_{n\to\infty}f_n:T\times T\to T$ {\bf \complim}
if, for any sequences $ a_{n},b_{n}, n\geq 1$ in $T$ with $\lim\limits_{n\to \infty}a_{n}= a$ and $\lim\limits_{n\to\infty} b_{n}=b$, we have
\begin{equation}
\lim_{n\to\infty}f_n(a,b)(=\lim_{n\to\infty}f_n(\lim_{m\to \infty} a_m,\lim_{k\to \infty}b_k))=\lim_{n\to\infty}f_n(a_n,b_n).
\mlabel{eq:jx}
\end{equation}
\end{defn}

\begin{exam}~\cite{GGH}
Let $V$ be a linear metric space and $f_{n}:V\times V\rightarrow V, n\geq 1,$ be maps. If $f_n$ uniformly converges to a continuous function $f$, then $f$ is \complim.
\end{exam}

\begin{defn}\cite{GGH}
\begin{enumerate}
\item Let $(G,\cdot)$ be a topological group, and $\mfrakL_{\frac{1}{n}},\mfrakH_{\frac{1}{n}},\frakB:G\rightarrow G, n\geq 1$, be maps.
 We call $\big(G,(\mfrakL_{\frac{1}{n}},\mfrakH_{\frac{1}{n}})\big)$  a  {\bf $\lim\limits_{n\to \infty}(\mfrakL_{\frac{1}{n}},\mfrakH_{\frac{1}{n}})$-group}
if
\begin{enumerate}
\item $G$ is a $(\mfrakL_{\frac{1}{n}}, \mfrakH_{\frac{1}{n}})$-group for each $n\geq 1$;
\item the limit $\lim\limits_{n\to\infty}\mfrakH_{\frac{1}{n}}\Big(\mfrakL_{\frac{1}{n}}(a)
\mfrakL_{\frac{1}{n}}(b)\Big)$ uniquely exists for each $(a,b)\in G\times G$; and \item the map $$\lim\limits_{n\to\infty}\mfrakH_{\frac{1}{n}}\Big(\mfrakL_{\frac{1}{n}}\cdot
\mfrakL_{\frac{1}{n}}\Big): G\times G \rightarrow G, \quad (a,b)\mapsto \lim\limits_{n\to\infty}\mfrakH_{\frac{1}{n}}\Big(\mfrakL_{\frac{1}{n}}(a)
\mfrakL_{\frac{1}{n}}(b)\Big)$$
is \complim.
\end{enumerate}
\mlabel{it:lhnrb1}
\item
Let $\frakg$ be a topological Lie algebra and $\MfrakL_{\frac{1
}{n}},\MfrakH_{\frac{1}{n}}, B:\frakg \rightarrow \frakg, n\geq 1,$ be linear maps.
We call $\big(\frakg,\big(\MfrakL_{\frac{1}{n}},\MfrakH_{\frac{1}{n}}\big)\big)$  a {\bf $\lim\limits_{n\to\infty}(\MfrakL_{\frac{1}{n}},\MfrakH_{\frac{1}{n}})$-Lie algebra} if
\begin{enumerate}
\item $\frakg$ is a $(\MfrakL_{\frac{1}{n}}, \MfrakH_{\frac{1}{n}})$-Lie algebra for each $n\geq 1$;
\item the limit $\lim\limits_{n\to\infty}\MfrakH_{\frac{1}{n}}\Big([\MfrakL_{\frac{1}{n}}(u),\MfrakL_{\frac{1}{n}}(v)]\Big)$ for $(u,v)\in \frakg\times \frakg,$ uniquely exists; and \item the map $$\lim\limits_{n\to\infty}\MfrakH_{\frac{1}{n}}\Big([\MfrakL_{\frac{1}{n}},\MfrakL_{\frac{1}{n}}]\Big):\frakg \times \frakg \rightarrow \frakg, \quad (u,v)\mapsto
\lim\limits_{n\to\infty}\MfrakH_{\frac{1}{n}}\Big([\MfrakL_{\frac{1}{n}}(u),\MfrakL_{\frac{1}{n}}(v)]\Big)$$
is \complim.
\end{enumerate}
\end{enumerate}
\end{defn}

Let $G$ be a Lie group and $\frakg$ its Lie algebra. We equip a topology on $\frakg$ by taking the open sets of $\frakg$ to be the inverse images of the open sets of $G$ under the map $\exp:\frakg\to G$, called the {\bf topology on $\frakg$ induced from its Lie group $G$}.

\begin{lemma}~\cite{GGH}
\mlabel{lemma:groupyr}
Let $G$ be a $\lim\limits_{n\to\infty}(\mfrakL_{\frac{1}{n}},\mfrakH_{\frac{1}{n}})$-Lie group with $(\mfrakL_{\frac{1}{n}},\mfrakH_{\frac{1}{n}})$ a unital pair for each $n\geq 1$. Let $\frakg = T_e G$ be the Lie algebra of $G$, and let $
 \MfrakL_{\frac{1}{n}}:=(\mfrakL_{\frac{1}{n}})_{*e}$,  $\MfrakH_{\frac{1}{n}}:=(\mfrakH_{\frac{1}{n}})_{*e}$
be the tangent maps of $\mfrakL_{\frac{1}{n}}$, $\mfrakH_{\frac{1}{n}}$ at the identity $e$, respectively. Then with the induced topology from $G$, $\frakg$ is a $\lim\limits_{n\to\infty}(\MfrakL_{\frac{1}{n}},\MfrakH_{\frac{1}{n}})$-Lie algebra.
\end{lemma}

We now give a relative version of $\lim\limits_{n\to\infty}(\mfrakL_{\frac{1}{n}},\mfrakH_{\frac{1}{n}})$-groups and
$\lim\limits_{n\to\infty}(\MfrakL_{\frac{1}{n}},\MfrakH_{\frac{1}{n}})$-Lie algebras. In the following, for an action $\Gamma$ of a group $G$ on a group $H$, we denote
$$\RE:G\rightarrow \Aut(H), \quad  a\mapsto \Gamma_a.$$
Likewise, for an action of a Lie algebra $\frakg$ on a Lie algebra $\frakh$, we denote
$$\re:\frakg\rightarrow \Der(\frakh), \quad u\mapsto \gamma_u.$$

\begin{defn}
\begin{enumerate}
\item Let $H$ be a $\lim\limits_{n\to\infty}(\mfrakL_{\frac{1}{n}},\mfrakH_{\frac{1}{n}})$-group and $G$ a group with a group action $\Gamma:G\to \Aut(H)$.
  We call $H$ an {\bf  action-\complim $\lim\limits_{n\to\infty}(\mfrakL_{\frac{1}{n}},\mfrakH_{\frac{1}{n}})$-group} for the action $\RE$ if, for each $a\in G$ and $b \in H$, the limit
$$\lim\limits_{n\to\infty}\mfrakH_{\frac{1}{n}}\RE_{a}\Big(\big(\mfrakL_{\frac{1}{n}}(b)\big)\Big)$$
uniquely exists and, for each sequence $b_n, n\geq 1,$ in $H$ with $\lim\limits_{n\to\infty}b_n=b$, there is
$$\lim\limits_{n\to\infty}\mfrakH_{\frac{1}{n}}\RE_{a}\bigg(\Big(\mfrakL_{\frac{1}{n}}(b)\Big)\bigg)=\lim\limits_{n\to\infty}\mfrakH_{\frac{1}{n}}\RE_{a}\bigg(\Big(\mfrakL_{\frac{1}{n}}(b_n)\Big)\bigg).$$
\item Let $\frakh$ be a  $\lim\limits_{n\to\infty}(\MfrakL_{\frac{1}{n}},\MfrakH_{\frac{1}{n}})$-Lie algebra and $\frakg$ a Lie algebra  with a Lie algebra action $\re:\frakg \rightarrow \Der(\frakh)$. We call $\frakh$ an {\bf action-\complim $\lim\limits_{n\to\infty}(\MfrakL_{\frac{1}{n}},\MfrakH_{\frac{1}{n}})$-Lie algebra} for the action $\re$ if, for each $u\in \frakg$ and $v \in \frakh$, the limit
$$\lim\limits_{n\to\infty}\MfrakH_{\frac{1}{n}}\Big(\re_{u}\big(\MfrakL_{\frac{1}{n}}(v)\big)\Big)$$
unique exists and, for each sequence $v_n, n\geq 1$ in $\frakh$ with $\lim\limits_{n\to\infty}v_n=v$, there is
$$\lim\limits_{n\to\infty}\MfrakH_{\frac{1}{n}}\bigg(\re_{u}\Big(\MfrakL_{\frac{1}{n}}(v)\Big)\bigg)=\lim\limits_{n\to\infty}\MfrakH_{\frac{1}{n}}\bigg(\re_{u}\Big(\MfrakL_{\frac{1}{n}}(v_n)\Big)\bigg).$$
\end{enumerate}
\end{defn}

There is a relation between action-\complim Lie groups and action-\complim Lie algebras as follows.

\begin{prop}
Let $G$ be a Lie group and $H$ be an action-\complim $\lim\limits_{n\to\infty}(\mfrakL_{\frac{1}{n}},\mfrakH_{\frac{1}{n}})$-group for a smooth action $\RE:G\rightarrow \Aut(H)$ where $(\mfrakL_{\frac{1}{n}},\mfrakH_{\frac{1}{n}})$ are unital pairs. Let $\frakh = T_e H$ be the Lie algebra of $H$, and let $
 \MfrakL_{\frac{1}{n}}:=(\mfrakL_{\frac{1}{n}})_{*e}$ and $\MfrakH_{\frac{1}{n}}:=(\mfrakH_{\frac{1}{n}})_{*e}$
be the tangent maps of $\mfrakL_{\frac{1}{n}}$ and $\mfrakH_{\frac{1}{n}}$ at the identity $e$, respectively. Then with the induced topology from $H$, $\frakh$ is an action-\complim $\lim\limits_{n\to\infty}(\MfrakL_{\frac{1}{n}},\MfrakH_{\frac{1}{n}})$-Lie algebra for the action $\re:=\RE_{*e}:\frakg\rightarrow \Der(\frakh)$.
\mlabel{prop:leiswl}
\end{prop}
\begin{proof}
For each $v\in \frakh$ and $a\in G$, we have
{\small
\begin{align*}
\left.\dfrac{\dd^2}{\dd t\dd s}\right|_{t,s=0}\lim\limits_{n\to\infty}\mfrakH_{\frac{1}{n}}\Bigg(\RE_{e^{tu}}
\bigg(\mfrakL_{\frac{1}{n}}(e^{sv})\bigg)\Bigg)
= \lim_{n\to\infty}(\mfrakH_{\frac{1}{n}})_{*e}\bigg(\left.\dfrac{\dd^2}{\dd t\dd s}\right|_{t,s=0}a\RE_{e^{tu}}\bigg(\mfrakL_{\frac{1}{n}}(e^{sv})\bigg)\bigg)= \lim\limits_{n\to\infty}\MfrakH_{\frac{1}{n}}\bigg(\re_{u}\MfrakL_{\frac{1}{n}}(v)\bigg),
\end{align*}}
and for each sequence $v_{n}, n\geq 1$ in $\frakh$ with $\lim\limits_{n\to\infty}v_n=v$, we have
{\small
\begin{align*}
&\ \left.\dfrac{\dd^2}{\dd t\dd s}\right|_{t,s=0}\lim_{n\to\infty}\mfrakH_{\frac{1}{n}}\Bigg(\RE_{e^{tu}}\bigg(\mfrakL_{\frac{1}{n}}(e^{sv})\bigg)\Bigg) =\left.\dfrac{\dd^2}{\dd t\dd s}\right|_{t,s=0}\lim\limits_{n\to\infty}\mfrakH_{\frac{1}{n}}\Bigg(\RE_{e^{tu}}\bigg(\mfrakL_{\frac{1}{n}}(e^{sv_n})\bigg)\Bigg) \\
=&\ \lim_{n\to\infty}(\mfrakH_{\frac{1}{n}})_{*e}\Bigg(\left.\dfrac{\dd^2}{\dd t\dd s}\right|_{t,s=0}\RE_{e^{tu}}\bigg(\mfrakL_{\frac{1}{n}}(e^{sv_n})\bigg)\Bigg)
= \lim\limits_{n\to\infty}\MfrakH_{\frac{1}{n}}
\bigg(\re_{u}\Big(\MfrakL_{\frac{1}{n}}(v_n)\Big)\bigg). \qedhere
\end{align*}}
\end{proof}

\subsection{Limit-weighted relative Rota-Baxter operators}
We now introduce our main notions on relative Rota-Baxter operators with limit-weights on groups.

\begin{defn}
\begin{enumerate}
\item
Let $(G,\cdot_{G})$ be a group and $(H,\cdot_{H})$ be an action-\complim $\lim\limits_{n\to\infty}(\mfrakL_{\frac{1}{n}},\mfrakH_{\frac{1}{n}})$-group for an action $\RE:G\rightarrow \Aut(H)$. A map $\frakB:H\rightarrow G$ is called a {\bf relative Rota-Baxter operator with limit-weight $\lim\limits_{n\to\infty}(\mfrakL_{\frac{1}{n}},\mfrakH_{\frac{1}{n}})$ on $(G,\cdot_{G})$ for the action $\RE$} if
\begin{equation}
	\frakB(a)\cdot_{G}\frakB(b)=\frakB\bigg(\lim_{n\to\infty}
	\mfrakH_{\frac{1}{n}}\bigg(\mfrakL_{\frac{1}{n}}(a)
	\cdot_{H} \RE_{\frakB(a)} \Big(\mfrakL_{\frac{1}{n}}(b)\Big) \bigg)\bigg), \quad a,b\in H.
	\mlabel{eq:relrbozero}
\end{equation}
We then call $(H, G, \Gamma, \frakB)$ a {\bf relative Rota-Baxter group with limit-weight $\lim\limits_{n\to\infty}(\mfrakL_{\frac{1}{n}},\mfrakH_{\frac{1}{n}})$}.
\mlabel{it:relrblimit}
\item
A relative Rota-Baxter operator $\frakB:H\to G$ with limit-weight $\lim\limits_{n\to\infty}(\mfrakL_{\frac{1}{n}},\mfrakH_{\frac{1}{n}})$ that satisfies  $\lim\limits_{n\to\infty}\mfrakL_{\frac{1}{n}}(a)=e$ for $a\in G$ is called a {\bf relative Rota-Baxter operator with \limwtzero}. We then call $(H, G, \Gamma, \frakB)$ a {\bf relative Rota-Baxter group  with \limwtzero}.
\mlabel{it:relrblimzero}
\end{enumerate}
\mlabel{defn:relrb}
\end{defn}

\begin{exam}
\begin{enumerate}
\item
With respect to the adjoint action, a relative Rota-Baxter operator with limit-weight or limit-weight zero is reduced to the corresponding notion of (non-relative) Rota-Baxter operator in~\mcite{GGH}.
\item
When all the maps $\mfrakL_{\frac{1}{n}}$ and $\mfrakH_{\frac{1}{n}}, n\geq 1$, are taken to be the identity map on $H$ in a relative Rota-Baxter operator with limit-weight $\lim\limits_{n\to\infty}(\mfrakL_{\frac{1}{n}},\mfrakH_{\frac{1}{n}})$ on a group $H$, we recover the relative Rota-Baxter operator (with weight one) on $H$ in~\mcite{LST1}.
\end{enumerate}
\end{exam}

We now define the descent group from a relative Rota-Baxter operator with limit-weight.

\begin{theorem}
\mlabel{thm:rbdes1}
Let $ \frakB:H\rightarrow G$ be a relative Rota-Baxter operator with limit-weight $\lim\limits_{n\to\infty}(\mfrakL_{\frac{1}{n}},\mfrakH_{\frac{1}{n}})$ on a group $(G,\cdot_{G})$ for an action $\RE:G\rightarrow \Aut(H)$.
\begin{enumerate}
\item With the multiplication
\begin{equation*}
\mlabel{eq:nast}
a \ast b := \lim_{n\to \infty}\mfrakH_{\frac{1}{n}}\Big(\mfrakL_{\frac{1}{n}}(a) \RE_{\frakB\left(a\right)}\mfrakL_{\frac{1}{n}}(b) \Big), \quad  a,b\in H,
\end{equation*}
$H$ is a semigroup.
\mlabel{it:desgroup1}

\item If $(\mfrakL_{\frac{1}{n}},\mfrakH_{\frac{1}{n}})$ is a bijective pair for each $n\geq 1$, and
if the limit
$\lim\limits_{n \to \infty }\mfrakH_{\frac{1}{n}}(e) $
uniquely exists, then $\big(H,*)$ is a group, called the {\bf descent group} of $(H, \frakB)$.
\mlabel{it:desgroup2}
\end{enumerate}
\end{theorem}

\begin{proof}
(\mref{it:desgroup1})
We just need to check the associativity of $\ast$: for $a, b, c\in H$, we have
{\small
\begin{align*}
 (a \ast b ) \ast c
=&\  \bigg( \lim_{n\to \infty} \Big( \mfrakL_{\frac{1}{n}}(a)\RE_{\frakB\left(a\right)}(\mfrakL_{\frac{1}{n}}(b) ) \Big)  \bigg) \ast c\\
=&\ \lim_{\substack{n=m \to \infty}}  \mfrakH_{\frac{1}{m}} \Bigg(\mfrakL_{\frac{1}{m}} \bigg(\mfrakH_{\frac{1}{n}} \Big( \mfrakL_{\frac{1}{n}}(a) \RE_{\frakB\left(a\right) }(\mfrakL_{\frac{1}{n}}(b) ) \Big)  \bigg) \RE_{\frakB\left(a\right)\frakB\left(b\right)}(\mfrakL_{\frac{1}{m}}(c)  )\Bigg) \quad \quad \text{(by Eq.~\meqref{eq:jx})}\\
=&\ \lim_{n\to \infty}\mfrakH_{\frac{1}{n}} \Big( \mfrakL_{\frac{1}{n}}(a) \RE_{\frakB\left(a\right)}\mfrakL_{\frac{1}{n}}(b) \RE_{\frakB\left(a\right)\frakB\left(b\right)}(c) \Big)\\
=&\ \lim_{\substack{n=m \to \infty}} \mfrakH_{\frac{1}{m}} \Bigg( \mfrakL_{\frac{1}{m}}(a) \RE_{\frakB\left(a\right)}\bigg( \mfrakL_{\frac{1}{m}} \Big(\mfrakH_{\frac{1}{n}} \Big( \mfrakL_{\frac{1}{n}}(b) \RE_{\frakB\left(b\right)}(\mfrakL_{\frac{1}{n}}(c))\Big)  \Big ) \bigg)\Bigg) \\
=&\ a \ast \mfrakH_{\frac{1}{n}}  \Big(\mfrakL_{\frac{1}{n}}(b) \RE_{\frakB\left(b\right)}\mfrakL_{\frac{1}{n}}(c)  \Big)   \hspace{1cm}\text{(by Eq.~\meqref{eq:jx})}\\
=&\ a \ast  (b \ast c ).
\end{align*}}

\noindent
(\mref{it:desgroup2})
We first have
{\small
\begin{align*}
    \frakB \big(\lim\limits_{n \to \infty }\mfrakH_{\frac{1}{n}}(e)\big)\frakB \big(\lim\limits_{n \to \infty }\mfrakH_{\frac{1}{n}}(e)\big)=&\ \frakB\bigg(\lim_{n\to\infty}\mfrakH\Big(\mfrakL_{\frac{1}{n}}(\mfrakH_{\frac{1}{n}}(e))\mfrakL_{\frac{1}{n}}(\mfrakH_{\frac{1}{n}}(e))\Big)\bigg)\ \quad \text{(by Eq~\meqref{eq:jx})}\\
    =&\ \frakB \big(\lim\limits_{n \to \infty }\mfrakH_{\frac{1}{n}}(e)\big),
\end{align*}
}
so we get $\frakB \big(\lim\limits_{n \to \infty }\mfrakH_{\frac{1}{n}}(e)\big)=e$.
Further,
{\small
\begin{align*}
    a \ast \lim_{n\to \infty}\mfrakH_{\frac{1}{n}}(e)
    =&\ \lim_{\substack{n=m \to \infty}}\mfrakH_{\frac{1}{m}}\bigg(\mfrakL_{\frac{1}{m}}(a) \RE_{\frakB(a)}\Big( \mfrakL_{\frac{1}{m}}(\mfrakH_{\frac{1}{n}}(e)) \Big) \bigg) \quad \text{(by Eq.~\meqref{eq:jx})}\\
    =&\ \lim_{\substack{n \to \infty}}\mfrakH_{\frac{1}{n}}\bigg(\mfrakL_{\frac{1}{n}}(a) \RE_{\frakB(a)}\Big( \mfrakL_{\frac{1}{n}}(\mfrakH_{\frac{1}{n}}(e))\Big) \bigg)= \lim_{n\to\infty}\mfrakH_{\frac{1}{n}}(\mfrakL_{\frac{1}{n}}(a))= a,
\end{align*}}
and
{\small
\begin{align*}
&    a \ast \lim\limits_{n \to \infty }\mfrakH_{\frac{1}{n}} \Big(\RE_{\frakB (a )^{-1}}(\mfrakL_{\frac{1}{n}}(a)^{-1}) \Big)\\
    =&\ \lim_{\substack{n=m \to \infty}} \mfrakH_{\frac{1}{m}} \Bigg(\mfrakL_{\frac{1}{m}}(a)\RE_{\frakB (a )}\Bigg(\mfrakL_{\frac{1}{m}} \bigg(\mfrakH_{\frac{1}{n}} \Big(\RE_{\frakB (a )^{-1}}(\mfrakL_{\frac{1}{n}}(a)^{-1}) \Big) \bigg)\Bigg) \Bigg)\quad \text{(by Eq.~\meqref{eq:jx})}\\
    =&\ \lim_{\substack{n \to \infty }} \mfrakH_{\frac{1}{n}} \Big(\mfrakL_{\frac{1}{n}}(a)\RE_{\frakB (a )\frakB (a )^{-1}}(\mfrakL_{\frac{1}{n}}(a)^{-1})\Big)
    = \lim_{n\to \infty}\mfrakH_{\frac{1}{n}}(e). \qedhere
\end{align*}}
\end{proof}

Examples and properties of a Rota-Baxter operator with limit-weight zero on a group can be found in~\mcite{GGH}. We give the Lie algebra version of the above notions.

\begin{defn}
\mlabel{defn:relie}
Let $(\frakg,[\cdot,\cdot]_{\frakg})$ be a Lie algebra and $(\frakh,[\cdot,\cdot]_{\frakh})$ an action-\complim $\lim\limits_{n\to\infty}(\MfrakL_{\frac{1}{n}},\MfrakH_{\frac{1}{n}})$-Lie  algebra for an action $\re:\frakg \rightarrow \Der(\frakh)$. A linear operator $B:\frakh\to \frakg$ is called a {\bf relative Rota-Baxter operator with limit-weight $(\MfrakL_{\frac{1}{n}},\MfrakH_{\frac{1}{n}})$ on $\frakg$ for the action $\re$} if
\begin{equation*}
[B(u),B(v)]=B\bigg(\lim_{n\to\infty}\MfrakH_{\frac{1}{n}}
\Big(\re_{B(u)}(\MfrakL_{\frac{1}{n}}(v))-\re_{B(v)}
(\MfrakL_{\frac{1}{n}}(u))+[\MfrakL_{\frac{1}{n}}(u),\MfrakL_{\frac{1}{n}}(v)]\Big)\bigg),
\quad u, v\in \frakh.
\end{equation*}
We then call $(\frakh, \frakg, \gamma, B)$ a {\bf relative Rota-Baxter Lie algebra with limit-weight $(\MfrakL_{\frac{1}{n}},\MfrakH_{\frac{1}{n}})$}.
\end{defn}

\begin{exam}
\begin{enumerate}
\item
In Definition~\mref{defn:relie}, taking
$\MfrakL_{\frac{1}{n}}:=\frac{1}{n} \id_{\frakh}$ (resp. $\MfrakL_{\frac{1}{n}}:=\lambda \id_{\frakh}$) and $\MfrakH_{\frac{1}{n}}:=n\, \id_{\frakh}$ (resp. $\MfrakH_{\frac{1}{n}}:=\frac{1}{\lambda} \id_{\frakh}$), a relative Rota-Baxter Lie algebra with limit-weight $\lim\limits_{n\to\infty}(\MfrakL_{\frac{1}{n}},\MfrakH_{\frac{1}{n}})$
reduces to the usual relative Rota-Baxter operator with weight zero (resp. weight $\lambda$).
\mlabel{rk:re weight lambda}
\item
In the special case when the Lie algebra $(\frakh,[\cdot,\cdot]_{\frakh})$ is $(\frakg,[\cdot,\cdot]_{\frakg})$ and the action $\re$ is the adjoint action of $\frakg$, we recover the notion of Rota-Baxter operators on Lie algebras with limit-weight from~\mcite{GGH}.
\mlabel{rk: reweight 0}
\end{enumerate}
\mlabel{rk:relrblie}
\end{exam}

The relative Rota-Baxter operators on Lie groups and those on Lie algebras are related as follows.

\begin{theorem}
Let $(G,\cdot_{G})$ be a Lie group and $(H,\cdot_{H})$ be an action-\complim $\lim\limits_{n\to\infty}(\mfrakL_{\frac{1}{n}},\mfrakH_{\frac{1}{n}})$-group for a smooth action $\RE:G\rightarrow \Aut(H)$, such that $\mfrakL_{\frac{1}{n}}$ and $\mfrakH_{\frac{1}{n}}$, $n\geq 1$, are smooth and unital maps. Let $\frakB:H\rightarrow G$ be a smooth relative Rota-Baxter operator with limit-weight $\lim\limits_{n\to\infty}(\mfrakL_{\frac{1}{n}},\mfrakH_{\frac{1}{n}})$ for the action $\RE$.  Let $\frakg = T_e G$ be the Lie algebra of $G$ where the topology on $\frakg$ is induced from $G$, and let
$$
B:=\frakB_{*e},\quad \MfrakL_{\frac{1}{n}}:=(\mfrakL_{\frac{1}{n}})_{*e},\quad \MfrakH_{\frac{1}{n}}:=(\mfrakH_{\frac{1}{n}})_{*e},
$$
be the tangent maps at $e$. Then $B:\frakh\rightarrow\frakg$ is a relative Rota-Baxter operator of limit-weight $\lim\limits_{n\to\infty}(\MfrakL_{\frac{1}{n}},\MfrakH_{\frac{1}{n}})$ for the action $\re:=\RE_{*}$ where the topology on $\frakh$ is induced from $H$.
\mlabel{thm:tgrelim}
\end{theorem}
\begin{proof}
It follows from Proposition~\mref{prop:leiswl} and
{\small
\begin{align*}
\left [ B(u), B(v) \right ]_g
=&\  \frac{\mathrm{d}^2}{\mathrm{d}t \mathrm{d}s}\bigg|_{t,s=0} e^{tB(u)}\cdot_G e^{sB(v)}\cdot_G e^{-tB(u)}\\
=&\  \frac{\mathrm{d}^2}{\mathrm{d}t \mathrm{d}s}\bigg|_{t,s=0} \frakB(e^{tu})\cdot_G \frakB(e^{sv})\cdot_G \frakB(e^{-tu})\\
=&\  \frac{\mathrm{d}^2}{\mathrm{d}t \mathrm{d}s}\bigg|_{t,s=0} \frakB(e^{tu})\cdot_G \frakB\Bigg(\lim_{n\to\infty} \mfrakH_{\frac{1}{n}}\bigg(\mfrakL_{\frac{1}{n}}(e^{sv})\cdot_H \RE_{\frakB(e^{sv})}\Big(\mfrakL_{\frac{1}{n}}(e^{-tu}) \Big) \bigg) \Bigg)\\
=&\  \frac{\mathrm{d}^2}{\mathrm{d}t \mathrm{d}s}\bigg|_{t,s=0} \frakB\Bigg(\lim_{n\to\infty}\mfrakH_{\frac{1}{n}} \bigg( \mfrakL_{\frac{1}{n}}(e^{tu}) \cdot_H \RE_{\frakB(e^{tu})}\Big(\mfrakL_{\frac{1}{n}}(e^{sv}) \Big) \cdot_H  \RE_{\frakB(e^{tu})\frakB(e^{sv})}\Big(\mfrakL_{\frac{1}{n}}(e^{-tu}) \Big)\bigg)\Bigg)\\
=&\ \lim_{n\to\infty} \frac{\mathrm{d}^2}{\mathrm{d}t \mathrm{d}s}\bigg|_{t,s=0} \frakB\mfrakH_{\frac{1}{n}} \bigg( \mfrakL_{\frac{1}{n}}(e^{tu}) \cdot_H \RE_{\frakB(e^{tu})}\Big(\mfrakL_{\frac{1}{n}}(e^{sv}) \Big) \cdot_H  \RE_{\frakB(e^{tu})\frakB(e^{sv})}\Big(\mfrakL_{\frac{1}{n}}(e^{-tu}) \Big)\bigg)\\
=&\    B\Bigg(\lim_{n\to\infty} \MfrakH_{\frac{1}{n}}\bigg(\re_{B(u)}\Big(\MfrakL_{\frac{1}{n}}(v)\Big)-\re_{B(v)}\Big(\MfrakL_{\frac{1}{n}}(u)\Big)+[\MfrakL_{\frac{1}{n}}(u),\MfrakL_{\frac{1}{n}}(v)]_{\frakh}\bigg) \Bigg), \quad u,v\in \frakh. \qedhere
\end{align*}  }
\end{proof}

In~\mcite{LST2}, a relative Rota-Baxter operator with weight zero is defined by taking $(H,\cdot_{H})$ to be an abelian group in a relative Rota-Baxter operator (with weight one).
Now we give a concrete example of relative Rota-Baxter operators with weight zero in this sense, which is also an example of relative Rota-Baxter operators with limit-weight zero in Definition~\mref{defn:relrb}~\meqref{it:relrblimzero}.

Recall that all the differentiable maps $\mathbb{R}\rightarrow \GL(n,\mathbb{C})$, equipped with the multiplication $\cdot$ induced from the matrix multiplication, form a group $(C^{1}(\RR,\GL(n,\mathbb{C})),\cdot)$.
Similarly, all the maps $\mathbb{R}\rightarrow \gl(n,\mathbb{C})$ equipped with the addition $+$ and the scalar-multiplication induced from the matrix addition and the matrix scalar-multiplication form a linear space and hence an additive group $(C(\mathbb{R},\gl(n,\mathbb{C})),+)$.
Define a group action $$C^{1}(\mathbb{R},\GL(n,\mathbb{C}))\rightarrow \Aut\Big(C(\mathbb{R},\gl(n,\mathbb{C}))\Big),\quad a\mapsto (b\mapsto aba^{-1}).$$
Here $aba^{-1}:\RR\to \gl(n,\mathbb{C}), \, x\mapsto a(x)b(x)a(x)^{-1}$.

Let $I\in \gl(n,\mathbb{C})$ be the identity matrix. For each $u \in C(\RR,\gl(n,\mathbb{C})) $, the initial-value problem of the matrix differential equation
\begin{align}
\left\{ \begin{array}{ll}\frac{d}{dx}f(x)=\ u(x)f(x),& \\
\,\,\,\,\, f(0)=I,&
\end{array}
\right.
\mlabel{eq:inival}
\end{align}
has a unique solution $f_{u}\in C^{1}(\RR,\GL(n,\mathbb{C})).$ Thus we obtain a map
\begin{equation}
\frakS: C(\RR,\gl(n,\mathbb{C})) \to C^{1}(\RR,\GL(n,\mathbb{C})), \quad u\mapsto f_{u}.
\mlabel{eq:smap}
\end{equation}

\begin{prop}
In Definition~\mref{defn:relrb}, take $$(G,\cdot_{G}):=\Big(C^{1}(\RR,\GL(n,\mathbb{C})),\cdot\Big)\,\text{ and }\, (H,\cdot_{H}):=\big(C(\RR,\gl(n,\mathbb{C})),+\big).$$
\begin{enumerate}
	\item
\mlabel{it:rbzeroexam}
The map $\frakS$ defined in Eq.~\meqref{eq:smap} is a relative Rota-Baxter operator with weight zero, that is, for $u,v\in C(\RR,\gl(n,\mathbb{C}))$, we have
\begin{align}
\mlabel{eq:rbivp}
        \mathfrak{S}(u)\mathfrak{S}(v)=\mathfrak{S}
        \Big(u+\mathfrak{S}\big(u\big)v\mathfrak{S}\big(u\big)^{-1}\Big).
\end{align}
\item
Take $(\mfrakL_{\frac{1}{n}},\mfrakH_{\frac{1}{n}}):=(\frac{1}{n}\id_H,n\,\id_H), n\geq 1$.
Then $\frakS$ is a relative Rota-Baxter operator with limit-weight zero.
\mlabel{it:rblimitzeroexam}
\end{enumerate}
\mlabel{prop:dfbeez}
\end{prop}
\begin{proof}
\meqref{it:rbzeroexam} We just need to verify Eq.~\meqref{eq:rbivp}, as follows. First differentiating the left hand side, we obtain
\vspace{-.2cm}
\begin{align*}
 \frac{d}{dx}\mathfrak{S}(u)\mathfrak{S}(v)(x)&\ = \Big(u\mathfrak{S}(u)\mathfrak{S}(v)+\mathfrak{S}(u)v\mathfrak{S}(v)\Big)(x)\\
&\ = \Big(u\mathfrak{S}(u)\mathfrak{S}(v) +\mathfrak{S}(u)v\mathfrak{S}(u)^{-1}\mathfrak{S}(u)\mathfrak{S}(v)\Big)(x)\\
&\ =\Big((u+\mathfrak{S}(u)v\mathfrak{S}(u)^{-1})
\mathfrak{S}(u)\mathfrak{S}(v)\Big)(x).
\end{align*}
So $\mathfrak{S}(u)\mathfrak{S}(v)$ is a solution of the differential equation
\begin{align*}
\frac{d}{dx}h(x)=\Big((u +\mathfrak{S}(u)v
\mathfrak{S}(u)^{-1})h \Big)(x).
\end{align*}
By the uniqueness of the solution to Eq.~\meqref{eq:inival} with the given initial condition, we reach the conclusion.

\smallskip

\noindent
\meqref{it:rblimitzeroexam}
It is clear that $\lim\limits_{n\to\infty}\mfrakL_{\frac{1}{n}}(a)=e$ for $a \in G$.
Also Eq.~\meqref{eq:relrbozero} becomes
\vspace{-.2cm}
$$
\mathfrak{S}(u)\mathfrak{S}(v)=\mathfrak{S}\bigg(\lim_{n\to\infty}n
\bigg(\frac{1}{n}u+\mathfrak{S}(u)\frac{1}{n}v\mathfrak{S}(u)^{-1}\bigg)\bigg),
	\vspace{-.1cm}
$$
which simplifies to Eq.~\meqref{eq:rbivp}. Now the proof follows from Item~\meqref{it:rbzeroexam}.
\end{proof}

\begin{remark}
The relative Rota-Baxter operators with weight zero in~\mcite{LST2} mentioned above is different from the relative Rota-Baxter operators with limit-weight zero in Definition~\mref{defn:relrb}~\meqref{it:relrblimit}, where the group $(H,\cdot_{H})$ needs to be a $\lim\limits_{n\to\infty}(\mfrakL_{\frac{1}{n}},\mfrakH_{\frac{1}{n}})$-group. Under the assumption of  Proposition~\mref{prop:dfbeez}, it just happens that the two notions boil down to the same equation.
\end{remark}

An interpretation of Eq.~\meqref{eq:rbivp} is as follows. Suppose that we have found respectively solutions $\frakS(u)$ and $\frakS(v)$ to the initial-value problems of matrix differential equations
\vspace{-.2cm}
\[
\left\{ \begin{array}{ll}\frac{d}{dx}f(x)=\ u(x)f(x), \\
\,\,\,\,\, f(0)=I,
\end{array}
\right.
\text{and }\,\,\, \left\{ \begin{array}{ll}\frac{d}{dx}g(x)=\ v(x)g(x),& \\
\,\,\,\,\, g(0)=I,&
\end{array}
\right.
\]
Then we obtain a solution $\mathfrak{S}(u)\mathfrak{S}(v)$ to the initial-value problem
\[
\left\{ \begin{array}{ll}\frac{d}{dx}h(x)=\Big((u +\mathfrak{S}(u)v
\mathfrak{S}(u)^{-1})h \Big)(x), \\
\,\,\,\,\, h(0)=I.
\end{array}
\right.
\vspace{-.2cm}
\]
\begin{remark}
    We can get a group structure on $C(\mathbb{R},\gl(n, \mathbb{C}))$ by Theorem~\mref{thm:rbdes1}.
\end{remark}

\vspace{-.5cm}
\section{Limit-weighted Rota-Baxter groups, pre-groups and Yang-Baxter equation}
\mlabel{sec:ybe}
In this section we give a general notion of a limit-weighted post-group whose limit-abelian case gives a new notion of a pre-group. Taking the tangent space, these two notions give a limit-weighted post-Lie algebra and a pre-Lie algebra respectively. On the other hand, these two notions are derived from a limited weighted relative Rota-Baxter group and from such a group under the limit-abelian condition respectively.
Then these groups are applied to produce set-theoretic-solutions of the Yang-Baxter equation.

\subsection{Pre-groups}
\mlabel{ss:pregp}

Let us begin with the following concept.

\begin{defn}
	For a $\lim\limits_{n\to \infty}(\mfrakL_{\frac{1}{n}},\mfrakH_{\frac{1}{n}})$-group $G$, define
	\begin{align}
		\cdot_{(\mfrakL_{\frac{1}{\infty}},\mfrakH_{\frac{1}{\infty}})}:G\times G &\ \rightarrow G, \quad (a,b)\mapsto \lim_{n\rightarrow \infty} \mfrakH_{\frac{1}{n}}(\mfrakL_{\frac{1}{n}}(a)\mfrakL_{\frac{1}{n}}(b)),
		\mlabel{eq:transmul}
	\end{align}
	called the {\bf(limit) transported multiplication induced by $\lim\limits_{n\to\infty}(\mfrakL_{\frac{1}{n}},\mfrakH_{\frac{1}{n}})$}. Further, if $(G,\cdot_{(\mfrakL_{\frac{1}{\infty}},\mfrakH_{\frac{1}{\infty}})})$ is a group, we call $(G,\cdot_{(\mfrakL_{\frac{1}{\infty}},\mfrakH_{\frac{1}{\infty}})})$ the {\bf(limit) transported group induced by $\lim\limits_{n\to\infty}(\mfrakL_{\frac{1}{n}},\mfrakH_{\frac{1}{n}})$}.
	\mlabel{defn:tsgp}
\end{defn}

For simplicity, when $(\mfrakL_n,\mfrakH_n), n\geq 1,$ are bijective pairs and hence $\mfrakH_n=\mfrakL_n^{-1}$, we abbreviate $\cdot_{(\mfrakL_{\frac{1}{\infty}},\mfrakH_{\frac{1}{\infty}})}$ by $\cdot_{\mfrakL_{\frac{1}{\infty}}}$.
To ensure that $(G,\cdot_{\mfrakL_{\frac{1}{\infty}}})$ is a group, we introduce the following condition.
\begin{defn}
A $\lim\limits_{n\to\infty}(\mfrakL_{\frac{1}{n}},\mfrakH_{\frac{1}{n}})$-group $G$ is called {\bf complete} if
\[ \lim\limits_{n\to\infty}\mfrakH_{\frac{1}{n}}(e)
\,\text{ and }\, \lim\limits_{n\to\infty}\mfrakH_{\frac{1}{n}}(\mfrakL_{\frac{1}{n}}(a)^{-1}), \quad a\in G,
\]
uniquely exist.
\end{defn}

We expose a sufficient condition for a transported group.

\begin{lemma}
Let $G$ be a $\lim\limits_{n\to\infty}(\mfrakL_{\frac{1}{n}},\mfrakH_{\frac{1}{n}})$-group. Then $(G,\cdot_{\big(\mfrakL_{\frac{1}{\infty}},\mfrakH_{\frac{1}{\infty}}\big)})$ is a semigroup.
Further, if $G$ is a complete $\lim\limits_{n\to\infty}(\mfrakL_{\frac{1}{n}},\mfrakH_{\frac{1}{n}})$-group with
$(\mfrakL_{\frac{1}{n}},\mfrakH_{\frac{1}{n}}), n\geq 1,$ bijective, then $(G,\cdot_{\mfrakL_{\frac{1}{\infty}}})$ is a group.
\mlabel{lem:lgp1}
\end{lemma}

\begin{proof}
	Let $a,b,c \in G$. By applying Eq.~\meqref{eq:jx} repeatedly, we verify the desired identities as follows:
{\small
	\begin{align*}
		(a\cdot_{\big(\mfrakL_{\frac{1}{\infty}},\mfrakH_{\frac{1}{\infty}}\big)}b)\cdot_{\big(\mfrakL_{\frac{1}{\infty}},\mfrakH_{\frac{1}{\infty}}\big)} c
		=&\   \lim_{\substack{n=m \to \infty}} \mfrakH_{\frac{1}{n}}\Bigg(\mfrakL_{\frac{1}{n}}\bigg(\mfrakH_{\frac{1}{m}}\Big(\mfrakL_{\frac{1}{m}}(a)\mfrakL_{\frac{1}{m}}(b)\Big)\bigg)\mfrakL_{\frac{1}{n}}(c)\Bigg)
		\\
		=&\ \lim_{\substack{n \to \infty}} \mfrakH_{\frac{1}{n}}\Big(\mfrakL_{\frac{1}{n}}(a)\mfrakL_{\frac{1}{n}}(b)\mfrakL_{\frac{1}{n}}(c)\Big) \\
		=&\ \lim_{\substack{n=m \to \infty}} \mfrakH_{\frac{1}{m}}\Bigg(\mfrakL_{\frac{1}{m}}(a)\mfrakL_{\frac{1}{m}}\bigg(\mfrakH_{\frac{1}{n}}\Big(\mfrakL_{\frac{1}{n}}(b)\mfrakL_{\frac{1}{n}}(c)\Big)\bigg)\Bigg)\\
		=&\  a\cdot_{\big(\mfrakL_{\frac{1}{\infty}},\mfrakH_{\frac{1}{\infty}}\big)}(b\cdot_{\big(\mfrakL_{\frac{1}{\infty}},\mfrakH_{\frac{1}{\infty}}\big)} c).
	\end{align*}
}
Thus $(G,\cdot_{\big(\mfrakL_{\frac{1}{\infty}},\mfrakH_{\frac{1}{\infty}}\big)})$ is a semigroup. Further, if $G$ is a complete $\lim\limits_{n\to\infty}(\mfrakL_{\frac{1}{n}},\mfrakH_{\frac{1}{n}})$-group with
	$(\mfrakL_{\frac{1}{n}},\mfrakH_{\frac{1}{n}})$ bijective for each $n\geq 1$, then
	\begin{align*}
		a\cdot_{\mfrakL_{\frac{1}{\infty}}}\lim\limits_{n\to\infty}\mfrakH_n(e)  =&\ \lim_{\substack{n=m \to \infty}}  \mfrakH_{\frac{1}{n}}\Big(\mfrakL_{\frac{1}{n}}(a)\mfrakL_{\frac{1}{m}}(\mfrakH_{\frac{1}{m}}(e))\Big)
		=\ \lim_{n\to\infty} \mfrakH_{\frac{1}{n}}(\mfrakL_{\frac{1}{n}}(a)) = a,\\
		a\cdot_{\mfrakL_{\frac{1}{\infty}}} \lim\limits_{n\to\infty}\mfrakH_{\frac{1}{n}}(\mfrakL_{\frac{1}{n}}(a)^{-1}) =&\
		\lim_{\substack{n=m \to \infty}} \mfrakH_{\frac{1}{m}}\Bigg(\mfrakL_{\frac{1}{m}}(a)\mfrakL_{\frac{1}{m}}\bigg(\mfrakH_{\frac{1}{n}}\Big(\mfrakL_{\frac{1}{n}}(a)^{-1}\Big)\bigg)\Bigg)=\lim_{n\to \infty}\mfrakH_{\frac{1}{n}}(e). \hspace{1.5cm} \qedhere
	\end{align*}
\end{proof}

To give an example of a $\lim\limits_{n\to\infty}(\mfrakL_{\frac{1}{n}},\mfrakH_{\frac{1}{n}})$-group that has a non-isomorphic transported group, we recall the following result.

\begin{exam} \cite[Corollary 3.12]{GGH}
	\mlabel{coro:0mult}
Let $G$ be a complex matrix Lie group with a complex matrix Lie algebra $\frakg$ and with $\exp:\frakg \rightarrow G$ bijective. Define
\begin{equation}
\mfrakH_{\frac{1}{n}}:=\pown:G \rightarrow G, \quad a\mapsto a^{n},
\mlabel{eq:powern}
\end{equation}
and thus $\mfrakL_{\frac{1}{n}}=\pown^{-1}$. Then $G$ is a complete  $\lim\limits_{n\to \infty}(\pown^{-1},\pown)$-group with the multiplication given by
	\begin{equation*}
		\exp(X)\cdot_{\frac{1}{\infty}}\exp(Y)=\exp(X+Y), \quad \exp(X),\exp(Y)\in G.
		\mlabel{eq:xinfy}
	\end{equation*}
\end{exam}

\begin{remark}
	When $(\mfrakL_n,\mfrakH_n)$ is a bijective pair for each $n \geq 1$, it is in general not true that $G$ and the transported group $(G,\cdot_{\mfrakL_{\frac{1}{\infty}}})$ induced by $\lim\limits_{n\to\infty}(\mfrakL_{\frac{1}{n}},\mfrakH_{\frac{1}{n}})$ are isomorphic. Example~\mref{coro:0mult} gives such a counterexample, where the group $G$ is not commutative but $(G,\cdot_{\mfrakL_{\frac{1}{\infty}}})$ is commutative. Hence the two groups cannot be isomorphic.
	\mlabel{lem:lgp0}
\end{remark}

Next we use the transported group from a limit-weighted Rota-Baxter operator on a group to give a relative Rota-Baxter operator.

\begin{theorem}
	Let $(G,\cdot,\frakB)$ be a Rota-Baxter group with limit-weight $\lim\limits_{n\to \infty}(\mfrakL_{\frac{1}{n}},\mfrakH_{\frac{1}{n}})$, with $(\mfrakL_{\frac{1}{n}},\mfrakH_{\frac{1}{n}})$ bijective, and let $(G,\cdot_{\mfrakL_{\frac{1}{\infty}}})$ be the transported group induced by $\lim\limits_{n\to\infty}(\mfrakL_{\frac{1}{n}},\mfrakH_{\frac{1}{n}})$. Then there is a group action $\RE$ define by
\begin{align}
\RE:	(G,\cdot)  \rightarrow \Aut\big(G,\cdot_{\mfrakL_{\frac{1}{\infty}}}\big),\quad a\mapsto   \RE_a\, \text{ with } \, \RE_{a}b:=\lim\limits_{n\to\infty}\mfrakH_{\frac{1}{n}}(a\mfrakL_{\frac{1}{n}}(b) a^{-1})\,\text{ for }\, b\in G, \mlabel{eq:gpaction}
\end{align}
    and $\frakB$ is a relative Rota-Baxter operator
    on $(G,\cdot)$ for the group action $\RE$.
\mlabel{thm:tsre}
\end{theorem}

\begin{proof} The effect of $\RE$ is a group action of $(G,\cdot)$ on $(G,\cdot_{\mfrakL_{\frac{1}{\infty}}})$ because
{\small
	\begin{align*}
		(\RE_{a}b)\cdot_{\mfrakL_{\frac{1}{\infty}}}(\RE_{a}c)&=\Big(\lim_{n\to\infty}\mfrakH_{\frac{1}{n}}(a\mfrakL_{\frac{1}{n}}(b) a^{-1})\Big)\cdot_{\mfrakL_{\frac{1}{\infty}}}\Big(\lim_{n\to\infty}\mfrakH_{\frac{1}{n}}(a\mfrakL_{\frac{1}{n}}(c) a^{-1})\Big)\\
		&=\lim_{n\to \infty}\mfrakH_{\frac{1}{n}}(\mfrakH_{\frac{1}{n}}(a\mfrakL_{\frac{1}{n}}(b) a^{-1}))\mfrakH_{\frac{1}{n}}(\mfrakH_{\frac{1}{n}}(a\mfrakL_{\frac{1}{n}}(x) a^{-1}))\ \ \text{(by Eq.~\meqref{eq:jx})}\\
		&=\lim_{n\to\infty}\mfrakH_{\frac{1}{n}}(a\mfrakL_{\frac{1}{n}}(b)\mfrakL(c)_{\frac{1}{n}} a^{-1})=\lim_{n\to\infty}\mfrakH_{\frac{1}{n}}(a\mfrakL_{\frac{1}{n}}(b\cdot_{\mfrakL_{\frac{1}{\infty}}}c)a^{-1})\\
		&=\RE_{a}(b\cdot_{\mfrakL_{\frac{1}{\infty}}}c)
	\end{align*}}
	and
 {\small
	\begin{align*}
		\RE_{a}(\RE_{b}c)&=\lim_{n\to\infty}\mfrakH_{\frac{1}{n}}(a\mfrakL_{\frac{1}{n}}(\mfrakH_{\frac{1}{n}}(b\mfrakL_{\frac{1}{n}}(c)b^{-1}))a^{-1})=\lim_{n\to\infty}\mfrakH_{\frac{1}{n}}(ab\mfrakL_{\frac{1}{n}}(c)(ab)^{-1}) = \RE_{ab}c.
	\end{align*}}
	Moreover, it follows from Eqs.~(\ref{eq:jx}) and~(\ref{eq:gpaction}) that
	\begin{align*}
		\frakB(a)\frakB(b)=& \ \frakB\Big(a\cdot_{\mfrakL_{\frac{1}{\infty}}}(\RE_{\frakB(a)}b) \Big)
		= \frakB\Big(\lim_{n\to\infty}\mfrakH_{\frac{1}{n}}
		(\mfrakL_{\frac{1}{n}}(a)\frakB(a)\mfrakL_{\frac{1}{n}}(b)\frakB(b)^{-1}) \Big). \qedhere
	\end{align*}
\end{proof}

Now we can use the transported group to give a new definition of a pre-group.
First recall (see \mcite{Bai} for example) that a {\bf (left) pre-Lie algebra} is a vector space $A$ endowed with a binary bilinear operation $\xsj$ satisfying
\begin{equation*}
	u \xsj(v\xsj w)-(u\xsj v)\xsj w  = v\xsj(u\xsj w) -(v\xsj u)\xsj w, \quad  u,v,w\in A.
	\mlabel{preLie1}
\end{equation*}
Also recall that $\xsj$ yields a Lie bracket on $A$ given by
\begin{equation}
	[u,v] :=u\xsj v- v\xsj u. \mlabel{preLie2}
\end{equation}

Now we give a notion of limit-weighted post-Lie algebras.
\begin{defn}
	\mlabel{defn:limpost}
	Let $\frakg$ be a $\lim\limits_{n\to\infty}(\MfrakL_{\frac{1}{n}},\MfrakH_{\frac{1}{n}})$-Lie algebra.
\begin{enumerate}
\item If
\begin{equation*}			\lim_{n\to\infty}\MfrakH_{\frac{1}{n}}\Big([\MfrakL_{\frac{1}{n}}(u),
\MfrakL_{\frac{1}{n}}(v)]\Big)=0, \quad u,v\in \frakg,
\end{equation*}
then $\frakg$ is called a {\bf  $\lim\limits_{n\to\infty}(\MfrakL_{\frac{1}{n}},\MfrakH_{\frac{1}{n}})$-abelian Lie algebra} or simply a {\bf limit-abelian Lie algebra}.
		
\item Let $\xsj:\frakg\otimes\frakg\rightarrow\frakg$ be  a multiplication such that for $u,v,w\in\frakg$, there are
\begin{equation}			u\xsj\lim_{n\to\infty}\MfrakH_{\frac{1}{n}}\Big([\MfrakL_{\frac{1}{n}}(v),\MfrakL_{\frac{1}{n}}(w)])\Big)=\lim_{n\to\infty}\MfrakH_{\frac{1}{n}}\Big([\MfrakL_{\frac{1}{n}}(u\xsj v),\MfrakL_{\frac{1}{n}}(w)]\Big)+\lim_{n\to\infty}\MfrakH_{\frac{1}{n}}\Big([\MfrakL_{\frac{1}{n}}(v),\MfrakL_{\frac{1}{n}}(u\xsj w)]\Big)
\mlabel{eq:postdk}
\end{equation}
and
\begin{equation}			\Big(\lim_{n\to\infty}\MfrakH_{\frac{1}{n}}\Big([\MfrakL_{\frac{1}{n}}(u),\MfrakL_{\frac{1}{n}}(v)]\Big)+u\xsj v-v\xsj u\Big)\xsj w= u \xsj(v\xsj w)-v\xsj(u\xsj w ).
\mlabel{eq:postlimit}
\end{equation}
		Then we call $(\frakg,[\cdot,\cdot],\xsj)$ a {\bf post-Lie algebra with limit-weight $\lim\limits_{n\to\infty}(\MfrakL_{\frac{1}{n}},\MfrakH_{\frac{1}{n}})$}.
	\end{enumerate}
\end{defn}

\begin{remark}
	\mlabel{rk:post-pre}
\begin{enumerate}
\item Taking $\MfrakL_{\frac{1}{n}}=\MfrakH_{\frac{1}{n}}=\id_\frakg$ in Definition~\mref{defn:limpost} recovers the notion of a post-Lie algebra.
\item From the definition, a post-Lie algebra with limit-weight $\lim\limits_{n\to\infty}(\MfrakL_{\frac{1}{n}},\MfrakH_{\frac{1}{n}})$ that is
$\lim\limits_{n\to\infty}(\MfrakL_{\frac{1}{n}},\MfrakH_{\frac{1}{n}})$-abelian is a pre-Lie algebra.
\end{enumerate}
\end{remark}

Recall that a Rota-Baxter Lie algebra with weight one gives rise to a post-Lie algebra; while a Rota-Baxter Lie algebra with weight zero gives rise to a pre-Lie algebra\mcite{Ag,BGN}. Here we prove a similar result for Rota-Baxter operators with limit-weights.

\begin{prop}
Let $(\frakg,[\cdot,\cdot])$ be an action-\complim $\lim\limits_{n\to\infty}(\MfrakL_{\frac{1}{n}},\MfrakH_{\frac{1}{n}})$-Lie algebra for a continuous action $\re:\frakg \rightarrow\Der(\frakg)$. Let $B$ be a relative Rota-Baxter operator with limit-weight $\lim\limits_{n\to\infty}(\MfrakL_{\frac{1}{n}},\MfrakH_{\frac{1}{n}})$ on $\frakg$. Define
	$$u\xsj v:=\lim\limits_{n\to\infty}\MfrakH_{\frac{1}{n}}\Big(\re_{B(u)}(\MfrakL_{\frac{1}{n}}(v))\Big), \quad u, v\in \frakg.$$
	Then $(\frakg,[\cdot,\cdot],\xsj)$ is a post-Lie algebra with limit-weight $\lim\limits_{n\to\infty}(\MfrakL_{\frac{1}{n}},\MfrakH_{\frac{1}{n}})$. If in addition, $\frakg$ is $\lim\limits_{n\to\infty}(\MfrakL_{\frac{1}{n}},\MfrakH_{\frac{1}{n}})$-abelian, then $(\frakg,[\cdot,\cdot],\xsj)$ is a pre-Lie algebra.
	\mlabel{prop:rbpre}
\end{prop}

\begin{proof}
	The first statement follows from
	{\small\begin{align*}
			&\ u\xsj\lim_{n\to\infty}\MfrakH_{\frac{1}{n}}\Big([\MfrakL_{\frac{1}{n}}(v),\MfrakL_{\frac{1}{n}}(w)]\Big)=\lim_{n=m\to\infty} \MfrakH_{\frac{1}{m}}\Bigg(\re_{B(u)}\bigg(\MfrakL_{\frac{1}{m}}\bigg(\MfrakH_{\frac{1}{n}}\Big([\MfrakL_{\frac{1}{n}}(v),\MfrakL_{\frac{1}{n}}(w)]\Big)\bigg)\Bigg)
			\quad\text{(by Eq.~\meqref{eq:jx})}\\
			=&\  \lim_{n\to\infty} \MfrakH_{\frac{1}{n}}\Bigg(\re_{B(u)}\Big([\MfrakL_{\frac{1}{n}}(v),\MfrakL_{\frac{1}{n}}(w)]\Big)\Bigg)\\
			=&\  \lim_{n\to\infty} \MfrakH_{\frac{1}{n}}\Bigg(\Big[\re_{B(u)}\Big(\MfrakL_{\frac{1}{n}}(v)\Big),\MfrakL_{\frac{1}{n}}(w)]\Big]+\Big[\MfrakL_{\frac{1}{n}}(v),\re_{B(u)}\Big(\MfrakL_{\frac{1}{n}}(w)]\Big)\Bigg)\\
			=&\  \lim_{n=m\to\infty} \MfrakH_{\frac{1}{n}}\Bigg(\Big[\MfrakL_{\frac{1}{n}}\MfrakH_{\frac{1}{m}}\Big(\re_{B(u)}(\MfrakL_{\frac{1}{m}}(v))\Big),\MfrakL_{\frac{1}{n}}(w)]\Big]+\Big[\MfrakL_{\frac{1}{n}}(v),\MfrakL_{\frac{1}{n}}\MfrakH_{\frac{1}{m}}(\re_{B(u)}(\MfrakL_{\frac{1}{m}}(w))\Big]\Bigg)\quad\text{(by Eq.~\meqref{eq:jx})}\\
			=&\ \lim_{n\to\infty}\MfrakH_{\frac{1}{n}}\Big([\MfrakL_{\frac{1}{n}}(u\xsj v),\MfrakL_{\frac{1}{n}}(w)]\Big)+\lim_{n\to\infty}\MfrakH_{\frac{1}{n}}\Big([\MfrakL_{\frac{1}{n}}(v),\MfrakL_{\frac{1}{n}}(u\xsj w)]\Big),
		\end{align*}
	}
	and
	{\small
		\begin{align*}
			&\ \bigg(\lim_{n\to\infty}\MfrakH_{\frac{1}{n}}\Big([\MfrakL_{\frac{1}{n}}(u),\MfrakL_{\frac{1}{n}}(v)]\Big)+u\xsj v-v\xsj u\bigg)\xsj w\\
			= &\ \lim_{n\to\infty} \MfrakH_{\frac{1}{n}}\bigg(\re_{B\Big(\MfrakH_{\frac{1}{n}}\big([\MfrakL_{\frac{1}{n}}(u),\MfrakL_{\frac{1}{n}}(v)]\big)+\MfrakH_{\frac{1}{n}}\big(\re_{B(u)}(\MfrakL_{\frac{1}{n}}(v))\big)-\MfrakH_{\frac{1}{n}}\big(\re_{B(v)}(\MfrakL_{\frac{1}{n}}(u))\big)\Big)} \MfrakL_{\frac{1}{n}}(w)\bigg)\\
			= &\ \lim_{m\to\infty} \MfrakH_{\frac{1}{n}}\bigg(\re_{B\Big(\lim\limits_{n\to\infty}\Big(\MfrakH_{\frac{1}{n}}\big([\MfrakL_{\frac{1}{n}}(u),\MfrakL_{\frac{1}{n}}(v)]\big)+\MfrakH_{\frac{1}{n}}\big(\re_{B(u)}(\MfrakL_{\frac{1}{n}}(v))\big)-\MfrakH_{\frac{1}{n}}\big(\re_{B(v)}(\MfrakL_{\frac{1}{n}}(u))\big)\Big)\Big)} \MfrakL_{\frac{1}{m}}(w)\bigg)\\
			& \hspace{5cm} \text{(by the limit-action-\complim property)}\\
			=&\ \lim_{m\to\infty} \MfrakH_{\frac{1}{m}}\Bigg(\re_{[B(u),B(v)]} \Big(\MfrakL_{\frac{1}{m}}(w)\Big)\Bigg)\\
			=&\ \lim_{m\to\infty} \MfrakH_{\frac{1}{m}}\Bigg(\re_{B(u)}\bigg(\re_{B(v)}\Big(\MfrakL_{\frac{1}{m}}(w)\Big) \bigg)-\re_{B(v)}\bigg(\re_{B(u)}\Big(\MfrakL_{\frac{1}{m}}(w)\Big) \bigg)\Bigg)\\
			=&\ \lim_{m=n\to\infty} \MfrakH_{\frac{1}{m}}\Bigg(\re_{B(u)}\Big(\MfrakL_{\frac{1}{m}}\MfrakH_{\frac{1}{n}}\Big(\re_{B(v)}(\MfrakL_{\frac{1}{n}}(w))\Big) \Big)-\re_{B(v)}\Big(\MfrakL_{\frac{1}{m}}\MfrakH_{\frac{1}{n}}\Big(\re_{B(u)}(\MfrakL_{\frac{1}{n}}(w))\Big) \Big)\Bigg)\\
			=&\ \lim_{n\to\infty} \MfrakH_{\frac{1}{n}}\Bigg(\re_{B(u)}\bigg(\MfrakL_{\frac{1}{m}}\Big(\lim_{n\to \infty}\MfrakH_{\frac{1}{n}}\Big(\re_{B(v)}(\MfrakL_{\frac{1}{n}}(w))\Big)\Big) \bigg)\Bigg)-\lim_{n\to\infty} \MfrakH_{\frac{1}{n}}\Bigg(\re_{B(v)}\bigg(\MfrakL_{\frac{1}{m}}\bigg(\lim_{n\to \infty}\MfrakH_{\frac{1}{n}}\Big(\re_{B(u)}(\MfrakL_{\frac{1}{n}}(w))\Big)\bigg) \bigg)\Bigg)  \\
			&\ \hspace{5cm}\text{(by Eq.~\meqref{eq:jx})}\\
			=&\ u \xsj(v\xsj w)-v\xsj(u\xsj w ).
	\end{align*}}
	Then the second statement follows since a $\lim\limits_{n\to\infty}(\MfrakL_{\frac{1}{n}},\MfrakH_{\frac{1}{n}})$-abelian post-Lie algebra is a pre-Lie algebra.
\end{proof}

\begin{remark}
   Take
$\MfrakL_{\frac{1}{n}}:=\frac{1}{n} \id$ and $\MfrakH_{\frac{1}{n}}:=n\, \id$. Then a Rota-Baxter Lie algebra with limit-weight $\lim\limits_{n\to\infty}(\MfrakL_{\frac{1}{n}},\MfrakH_{\frac{1}{n}})$ which is also a $\lim\limits_{n\to\infty}(\MfrakL_{\frac{1}{n}},\MfrakH_{\frac{1}{n}})$-abelian Lie algebra reduces to a Rota-Baxter Lie algebra with weight zero.
\end{remark}

We generalize the notion of post-groups~\mcite{BGST} to semigroups.
\begin{defn}
	A {\bf post-semigroup} is a semigroup $(M,\cdot)$ equipped with a multiplication $\rhd$ on $M$ such that
	\begin{enumerate}
		\item for each $a\in M$, the left multiplication
		\begin{equation*}
			L_{a}^{\rhd}:M\rightarrow M, \quad b\mapsto a\rhd b
		\end{equation*}
		is a homomorphism of the semigroup $(M,\cdot)$, that is,
		\begin{equation}
			a\rhd(b\cdot c)=(a\rhd b)\cdot(a\rhd c), \quad a,b,c\in M;
			\mlabel{postgp1}
		\end{equation}
		\item  the following ``weighted" associativity holds,
		\begin{equation}
			\Big(a\cdot(a\rhd b)\Big)\rhd c = a\rhd (b\rhd c), \quad a,b,c\in M.
			\mlabel{postgp2}
		\end{equation}
	\end{enumerate}
\end{defn}

The following concept will be needed in a moment.

\begin{defn}
	Let $(G,\cdot)$ be a complete $\lim\limits_{n\to\infty}(\mfrakL_{\frac{1}{n}},\mfrakH_{\frac{1}{n}})$-group with unital pairs $(\mfrakL_{\frac{1}{n}},\mfrakH_{\frac{1}{n}}), n\geq 1$. We call $G$ a {\bf $\lim\limits_{n\to\infty}(\mfrakL_{\frac{1}{n}},\mfrakH_{\frac{1}{n}})$-abelian group}
	if the (limit) transported multiplication $\cdot_{(\mfrakL_{\frac{1}{\infty}},\mfrakH_{\frac{1}{\infty}})}$ is commutative:
	$$a\cdot_{(\mfrakL_{\frac{1}{\infty}},\mfrakH_{\frac{1}{\infty}})}b=b\cdot_{(\mfrakL_{\frac{1}{\infty}},\mfrakH_{\frac{1}{\infty}})}a, \quad a, b\in G.$$
	Further, if $G$ is a Lie group and $\mfrakL_{\frac{1}{n}},\mfrakH_{\frac{1}{n}}$ are smooth maps, we call $G$ a {\bf $\lim\limits_{n\to\infty}(\mfrakL_{\frac{1}{n}},\mfrakH_{\frac{1}{n}})$-abelian Lie group}.
\end{defn}

\begin{lemma}
	Let $G$ be a $\lim\limits_{n\to\infty}(\mfrakL_{\frac{1}{n}},\mfrakH_{\frac{1}{n}})$-abelian Lie group with $(\mfrakL_{\frac{1}{n}},\mfrakH_{\frac{1}{n}}), n\geq 1,$ unital pairs and $\frakg$ its Lie algebra where the topology on $\frakg$ is induced from $G$, and let
	$    \MfrakL_{\frac{1}{n}}=(\mfrakL_{\frac{1}{n}})_{*e}$ and  $\MfrakH_{\frac{1}{n}}=(\mfrakH_{\frac{1}{n}})_{*e}
	$ be the tangent maps at the identity $e$. Then
	$\frakg$ is a $\lim\limits_{n\to\infty}(\MfrakL_{\frac{1}{n}},\MfrakH_{\frac{1}{n}})$-abelian Lie algebra.
	\mlabel{lemma:prec}
\end{lemma}
\begin{proof}
	Indeed, we have
	{\small \begin{align*}
			&\ \lim_{n\to\infty}\MfrakH_{\frac{1}{n}}\Big([\MfrakL_{\frac{1}{n}}(u),
			\MfrakL_{\frac{1}{n}}(v)]\Big)= \lim_{n\to\infty}(\mfrakH_{\frac{1}{n}})_{*e}
			\Bigg(\left.\dfrac{\dd^2}{\dd t\dd s}\right|_{t,s=0}\mfrakL_{\frac{1}{n}}(e^{tu})
			\mfrakL_{\frac{1}{n}}(e^{sv})-\left.\dfrac{\dd^2}{\dd t\dd s}\right|_{t,s=0}\mfrakL_{\frac{1}{n}}(e^{sv})
			\mfrakL_{\frac{1}{n}}(e^{tu})\Bigg)\\
			=&\ \lim_{n\to\infty}(\mfrakH_{\frac{1}{n}})_{*e}
			\Bigg(\left.\dfrac{\dd^2}{\dd t\dd s}\right|_{t,s=0}\mfrakL_{\frac{1}{n}}(e^{tu})
			\mfrakL_{\frac{1}{n}}(e^{sv})\Bigg)
			-\lim_{n\to\infty}(\mfrakH_{\frac{1}{n}})_{*e}
			\Bigg(\left.\dfrac{\dd^2}{\dd t\dd s}\right|_{t,s=0}\mfrakL_{\frac{1}{n}}(e^{sv})
			\mfrakL_{\frac{1}{n}}(e^{tu})\Bigg)\\
			=&\ \left.\dfrac{\dd^2}{\dd t\dd s}\right|_{t,s=0}\lim_{n\to\infty}\mfrakH_{\frac{1}{n}}
			\Big(\mfrakL_{\frac{1}{n}}(e^{tu})
			\mfrakL_{\frac{1}{n}}(e^{sv})\Big)-\left.\dfrac{\dd^2}{\dd t\dd s}\right|_{t,s=0}\lim_{n\to\infty}\mfrakH_{\frac{1}{n}}
			\Big(\mfrakL_{\frac{1}{n}}(e^{sv})
			\mfrakL_{\frac{1}{n}}(e^{tu})\Big)\\
			=&\ \left.\dfrac{\dd^2}{\dd t\dd s}\right|_{t,s=0}\lim_{n\to\infty}\mfrakH_{\frac{1}{n}}
			\Big(\mfrakL_{\frac{1}{n}}(e^{tu})
			\mfrakL_{\frac{1}{n}}(e^{sv})\Big)-\left.\dfrac{\dd^2}{\dd t\dd s}\right|_{t,s=0}\lim_{n\to\infty}\mfrakH_{\frac{1}{n}}
			\Big(\mfrakL_{\frac{1}{n}}(e^{tu})
			\mfrakL_{\frac{1}{n}}(e^{sv})\Big)=0 .\qedhere
		\end{align*}
	}
\end{proof}

\begin{remark}
Consequently, the tangent space of a limit-weighted limit-abelian relative Rota-Baxter group is a limit-weighted limit-abelian relative Rota-Baxter Lie algebra.
\mlabel{rk:tgab}
\end{remark}

\begin{defn}
\begin{enumerate}
\item
A $\lim\limits_{n\to\infty}(\mfrakL_{\frac{1}{n}},\mfrakH_{\frac{1}{n}})$-group $G$ is called a {\bf $\lim\limits_{n\to\infty}(\mfrakL_{\frac{1}{n}},\mfrakH_{\frac{1}{n}})$-binary adjoint \complim group} if for any sequences $a_n,b_n\in G, n\geq 1$ with $\lim\limits_{n\to\infty}a_n=a$ and $\lim\limits_{n\to\infty}a_n=a$, there is
\begin{align*}
    \lim\limits_{n\to\infty}\mfrakH_{\frac{1}{n}}(a\mfrakL_{\frac{1}{n}}(b)a^{-1})=\lim\limits_{n\to\infty}\mfrakH_{\frac{1}{n}}(a_n\mfrakL_{\frac{1}{n}}(b_n)a_n^{-1}).
\end{align*}
\item
A $\lim\limits_{n\to\infty}(\MfrakL_{\frac{1}{n}},\MfrakH_{\frac{1}{n}})$-Lie algebra $\frakg$ is called a {\bf $\lim\limits_{n\to\infty}(\MfrakL_{\frac{1}{n}},\MfrakH_{\frac{1}{n}})$-binary adjoint \complim Lie algebra}  if for any sequences $u_n,v_n\in G, n\geq 1$ with $\lim\limits_{n\to\infty}u_n=u$ and $\lim\limits_{n\to\infty}v_n=v$, there is
\begin{align*}
    \lim\limits_{n\to\infty}\MfrakH_{\frac{1}{n}}([u,\MfrakL_{\frac{1}{n}}(v)])=\lim\limits_{n\to\infty}\MfrakH_{\frac{1}{n}}([u_n,\MfrakL_{\frac{1}{n}}(v_n)]).
\end{align*}
\end{enumerate}
\end{defn}

\begin{prop}
\begin{enumerate}
\item
Let $G$ be a complete $\lim\limits_{n\to\infty}(\mfrakL_{\frac{1}{n}},\mfrakH_{\frac{1}{n}})$-binary adjoint \complim group with $(\mfrakL_{\frac{1}{n}},\mfrakH_{\frac{1}{n}})$ unital pairs and for each $a\in G$, $\lim\limits_{n\to\infty}\mfrakL_{\frac{1}{n}}(a)=e$. Then $G$ is a $\lim\limits_{n\to\infty}(\mfrakL_{\frac{1}{n}},\mfrakH_{\frac{1}{n}})$-abelian group.
\item
Let $\frakg$ be a $\lim\limits_{n\to\infty}(\MfrakL_{\frac{1}{n}},\MfrakH_{\frac{1}{n}})$-binary adjoint \complim Lie algebra and for each $u\in \frakg$, $\lim\limits_{n\to\infty}\MfrakL_{\frac{1}{n}}(u)=0$. Then $\frakg$ is a $\lim\limits_{n\to\infty}(\MfrakL_{\frac{1}{n}},\MfrakH_{\frac{1}{n}})$-abelian Lie algebra.
\end{enumerate}
\mlabel{prop:gpzht}
\end{prop}

\begin{proof}
We only proof the group case. Notice that for each $a,b\in G$,
{\small
	\begin{align*}
		&\ \lim_{n\to\infty}\mfrakH_{\frac{1}{n}}\Big(\mfrakL_{\frac{1}{n}}(a)
		\mfrakL_{\frac{1}{n}}(b)\mfrakL_{\frac{1}{n}}(a)^{-1}\mfrakL_{\frac{1}{n}}(b)^{-1}\Big)\\
		=&\ \bigg(\lim_{n\to\infty}\mfrakH_{\frac{1}{n}}\Big(\mfrakL_{\frac{1}{n}}(a)\mfrakL_{\frac{1}{n}}(b)\mfrakL_{\frac{1}{n}}(a)^{-1}\Big)\bigg)\cdot_{\big(\mfrakL_{\frac{1}{\infty}},\mfrakH_{\frac{1}{\infty}}\big)}\bigg(\lim_{n\to\infty}\mfrakH_{\frac{1}{n}}\Big(\mfrakL_{\frac{1}{n}}(b)^{-1}\Big)\bigg)\\
		=&\ \bigg(\lim_{n\to\infty}\mfrakH_{\frac{1}{n}}\Big(\lim_{m\to\infty}\mfrakL_{\frac{1}{m}}(a)\mfrakL_{\frac{1}{n}}(b)\mfrakL_{\frac{1}{m}}(a)^{-1}\Big)\bigg)\cdot_{\big(\mfrakL_{\frac{1}{\infty}},\mfrakH_{\frac{1}{\infty}}\big)}\bigg(\lim_{n\to\infty}\mfrakH_{\frac{1}{n}}\Big(\mfrakL_{\frac{1}{n}}(b)^{-1}\Big)\bigg)\\
		&\ \text{(by $\lim\limits_{n\to\infty}(\mfrakL_{\frac{1}{n}},\mfrakH_{\frac{1}{n}})$-binary adjoint \complim property)}\\
		=&\ \lim_{n\to\infty}\mfrakH_{\frac{1}{n}}\Big(\mfrakL_{\frac{1}{n}}(b)\Big)\cdot_{\big(\mfrakL_{\frac{1}{\infty}},\mfrakH_{\frac{1}{\infty}}\big)}\mfrakH_{\frac{1}{n}}\Big(\mfrakL_{\frac{1}{n}}(b)^{-1}\Big)\\
		=&\ \lim_{n\to\infty}\mfrakH_{\frac{1}{n}}(e)=e.
	\end{align*}
}
	So we have
 {\small
	\begin{align*}
		b \cdot_{\big(\mfrakL_{\frac{1}{\infty}},\mfrakH_{\frac{1}{\infty}}\big)} a = &\ \lim_{n\to\infty}\mfrakH_{\frac{1}{n}}\Big(\mfrakL_{\frac{1}{n}}(a)\mfrakL_{\frac{1}{n}}(b)\mfrakL_{\frac{1}{n}}(a)^{-1}\mfrakL_{\frac{1}{n}}(b)^{-1}\Big) \cdot_{\big(\mfrakL_{\frac{1}{\infty}},\mfrakH_{\frac{1}{\infty}}\big)} b \cdot_{\big(\mfrakL_{\frac{1}{\infty}},\mfrakH_{\frac{1}{\infty}}\big)} a\\
		=&\ \lim_{n\to\infty}\mfrakH_{\frac{1}{n}}\Big(\mfrakL_{\frac{1}{n}}(a)
		\mfrakL_{\frac{1}{n}}(b)\mfrakL_{\frac{1}{n}}(a)^{-1}
		\mfrakL_{\frac{1}{n}}(b)^{-1}\mfrakL_{\frac{1}{n}}(b)
		\mfrakL_{\frac{1}{n}}(a)\Big)\\
		=&\ \lim_{n\to\infty}\mfrakH_{\frac{1}{n}}\Big(\mfrakL_{\frac{1}{n}}(a)\mfrakL_{\frac{1}{n}}(b)\Big)\\
		=&\ a \cdot_{\big(\mfrakL_{\frac{1}{\infty}},\mfrakH_{\frac{1}{\infty}}\big)} b.
	\end{align*}}
	Thus $G$ is a $\lim\limits_{n\to\infty}(\mfrakL_{\frac{1}{n}},\mfrakH_{\frac{1}{n}})$-ableian group.
\end{proof}
\begin{remark}
  Proposition~\mref{prop:gpzht} shows that
	\begin{enumerate}
		\item  a Rota-Baxter group with limit-weight zero is a limit-weighted limit-abelian Rota-Baxter group when it is a complete $\lim\limits_{n\to\infty}(\mfrakL_{\frac{1}{n}},\mfrakH_{\frac{1}{n}})$-binary adjoint \complim group.
		\item a Rota-Baxter Lie algebra with limit-weight zero is a  limit-weighted limit-abelian Rota-Baxter Lie algebra when it is a  $\lim\limits_{n\to\infty}(\MfrakL_{\frac{1}{n}},\MfrakH_{\frac{1}{n}})$-binary adjoint \complim Lie algebra.
	\end{enumerate}
\end{remark}

\begin{defn}
Let $(G,\cdot)$ be a $\lim\limits_{n\to\infty}(\mfrakL_{\frac{1}{n}},\mfrakH_{\frac{1}{n}})$-group equipped with an additional multiplication $\rhd$, and let $\cdot_{(\mfrakL_{\frac{1}{\infty}},\mfrakH_{\frac{1}{\infty}})}$ be the transported multiplication given in Eq.~\meqref{eq:transmul}.
\begin{enumerate}
\item	We call $(G,\cdot,\rhd)$ a {\bf post-group with limit-weight $\lim\limits_{n\to\infty}(\mfrakL_{\frac{1}{n}},\mfrakH_{\frac{1}{n}})$} if $(G,\cdot_{(\mfrakL_{\frac{1}{\infty}},\mfrakH_{\frac{1}{\infty}})},\rhd)$ is a post-semigroup and  $a\rhd e=e$ , $e\rhd a=a$, for all $a\in G$.
\item If moreover $(G,\cdot)$ is a $\lim\limits_{n\to\infty}(\mfrakL_{\frac{1}{n}},\mfrakH_{\frac{1}{n}})$-abelian group, then we call $(G, \cdot, \rhd)$ a {\bf pre-group}.
\item In the above two items, if $(G,\cdot)$ is a Lie group and $\rhd,\mfrakH_{\frac{1}{n}},\mfrakH_{\frac{1}{n}}$ are smooth maps, then  $(G,\cdot,\rhd)$ is called a {\bf post-Lie group with limit-weight $\lim\limits_{n\to\infty}(\mfrakL_{\frac{1}{n}},\mfrakH_{\frac{1}{n}})$} and a {\bf pre-Lie group} respectively.
\end{enumerate}
	\mlabel{defn:pregroup}
\end{defn}

\begin{remark}
	A pre-group is not necessarily a commutative post group.
	A counterexample is given in Proposition~\mref{prop:rb33}.
\end{remark}

\begin{lemma}~\cite{BGST}
	Let $(G,\cdot,\rhd)$ be a post-group. Define $*$: $G\times G \rightarrow G, (a,b)\mapsto a\cdot (a\rhd b)$. Then
	\begin{enumerate}
		\item $(G,\ast)$ is a group.
		\item Let $\id_G: G\rightarrow G$ be the identity map. Then $\id_G$ is a relative Rota-Baxter operator on $(G,\ast)$ for the action $L^{\rhd}$.
	\end{enumerate}
	\mlabel{lem:prere}
\end{lemma}

The following result on pre-Lie groups justifies their name by establishing their relation with pre-Lie algebras.

\begin{theorem}
	Let $(G,\cdot,\rhd)$ be a post-Lie group with limit-weight $\lim\limits_{n\to\infty}(\mfrakL_{\frac{1}{n}},\mfrakH_{\frac{1}{n}})$ $($resp. a pre-Lie group$)$ where $(\mfrakL_{\frac{1}{n}},\mfrakH_{\frac{1}{n}})$ are unital, and $(\frakg, [\cdot,\cdot])$ the Lie algebra of $(G,\cdot)$ equipped with the topology induced from $G$. Define
	\begin{equation*}
		\xsj:\frakg\otimes\frakg\rightarrow\frakg, \quad u\xsj v:= L_{*e}^{\rhd}(u)(v)= \left.\dfrac{\dd}{\dd t}\right|_{t=0}
		L_{\exp(tu)}^{\rhd}(v)= \left.\dfrac{\dd^2}{\dd t\dd s}\right|_{t,s=0}
		L_{\exp(tu)}^{\rhd}(\exp(sv)).
	\end{equation*}
	Then $(\frakg,[\cdot,\cdot],\xsj)$ (resp. $(\frakg, \xsj)$) is a post-Lie algebra with limit-weight $\lim\limits_{n\to\infty}(\MfrakL_{\frac{1}{n}},\MfrakH_{\frac{1}{n}})$ $($resp. a pre-Lie algebra$)$.
	\mlabel{thm:tangprelie}
\end{theorem}

\begin{proof}
	Using the Baker-Campbell-Hausdorff formula, we have
	{\small\begin{align*}
			&\ u\xsj\lim_{n\to\infty}\MfrakH_{\frac{1}{n}}\Big([\MfrakL_{\frac{1}{n}}(v),\MfrakL_{\frac{1}{n}}(w)]\Big)\\
			=&\ \left.\dfrac{\dd}{\dd t}\right|_{t=0} \left.\dfrac{\dd}{\dd s}\right|_{s=0}
			\left.\dfrac{\dd}{\dd r}\right|_{r=0} \exp(tu)\rhd \lim_{n\to\infty}\exp\Big(s\MfrakH_{\frac{1}{n}}(v)+r\MfrakL_{\frac{1}{n}}(w)
			+\frac{1}{2}[s\MfrakH_{\frac{1}{n}}(v),r\MfrakL_{\frac{1}{n}}(w)]
			+\cdots\Big)\\
            &- \left.\dfrac{\dd}{\dd t}\right|_{t=0} \left.\dfrac{\dd}{\dd s}\right|_{s=0}
			\left.\dfrac{\dd}{\dd r}\right|_{r=0} \exp(tu)\rhd\lim_{n\to\infty} \exp\Big(s\MfrakH_{\frac{1}{n}}(w)+r\MfrakL_{\frac{1}{n}}(v)
			+\frac{1}{2}[s\MfrakH_{\frac{1}{n}}(w),r\MfrakL_{\frac{1}{n}}(v)]
			+\cdots\Big)\\
			= &\ \left.\dfrac{\dd}{\dd t}\right|_{t=0} \left.\dfrac{\dd}{\dd s}\right|_{s=0}
			\left.\dfrac{\dd}{\dd r}\right|_{r=0} \exp(tu)\rhd\Big(\lim_{n\to\infty} \exp(sv)\cdot_{\big(\mfrakL_{\frac{1}{\infty}},\mfrakH_{\frac{1}{\infty}}\big)}\exp(rw)\Big)\\
            &-\left.\dfrac{\dd}{\dd t}\right|_{t=0} \left.\dfrac{\dd}{\dd s}\right|_{s=0}
			\left.\dfrac{\dd}{\dd r}\right|_{r=0} \exp(tu)\rhd\Big(\lim_{n\to\infty} \exp(rw)\cdot_{\big(\mfrakL_{\frac{1}{\infty}},\mfrakH_{\frac{1}{\infty}}\big)}\exp(sv)\Big)\\
			=&\ \left.\dfrac{\dd}{\dd t}\right|_{t=0} \left.\dfrac{\dd}{\dd s}\right|_{s=0}
			\left.\dfrac{\dd}{\dd r}\right|_{r=0}\bigg(\lim_{n\to\infty} \Big(\exp(tu)\rhd\exp(sv)\Big)\cdot_{\big(\mfrakL_{\frac{1}{\infty}},\mfrakH_{\frac{1}{\infty}}\big)}\Big(\exp(tu)\rhd\exp(rw)\Big)\bigg)\\
            &- \left.\dfrac{\dd}{\dd t}\right|_{t=0} \left.\dfrac{\dd}{\dd s}\right|_{s=0}
			\left.\dfrac{\dd}{\dd r}\right|_{r=0}\bigg(\lim_{n\to\infty} \Big(\exp(tu)\rhd\exp(rw)\Big)\cdot_{\big(\mfrakL_{\frac{1}{\infty}},\mfrakH_{\frac{1}{\infty}}\big)}\Big(\exp(tu)\rhd\exp(sv)\Big)\bigg)\\
			=&\ \left.\dfrac{\dd}{\dd t}\right|_{t=0} \left.\dfrac{\dd}{\dd s}\right|_{s=0}\left.\dfrac{\dd}{\dd r}\right|_{r=0} \exp\Big(\MfrakH_{\frac{1}{n}}(tv\xsj sv)+\MfrakL_{\frac{1}{n}}(tu\xsj rw)
			+\frac{1}{2}[\MfrakH_{\frac{1}{n}}(tu\xsj sv),\MfrakL_{\frac{1}{n}}(tu\xsj rw)]
			+\cdots\Big)\\
            &- \left.\dfrac{\dd}{\dd t}\right|_{t=0} \left.\dfrac{\dd}{\dd s}\right|_{s=0}\left.\dfrac{\dd}{\dd r}\right|_{r=0} \exp\Big(\MfrakH_{\frac{1}{n}}(tv\xsj sv)+\MfrakL_{\frac{1}{n}}(tu\xsj rw)
			+\frac{1}{2}[\MfrakH_{\frac{1}{n}}(tu\xsj rw),\MfrakL_{\frac{1}{n}}(tu\xsj sv)]
			+\cdots\Big)\\
			=&\ \lim_{n\to\infty}\MfrakH_{\frac{1}{n}}\Big([\MfrakL_{\frac{1}{n}}(u\xsj v),\MfrakL_{\frac{1}{n}}(w)]\Big)+\lim_{n\to\infty}\MfrakH_{\frac{1}{n}}\Big([\MfrakL_{\frac{1}{n}}(v),\MfrakL_{\frac{1}{n}}(u\xsj w)]\Big).
	\end{align*}
}
Thus Eq.~\meqref{eq:postdk} holds. In addition, we have
	{\small
		\begin{align*}
			&\ u \xsj(v\xsj w)-v\xsj(u\xsj w)\\
			=&\ \left.\dfrac{\dd}{\dd t}\right|_{t=0} \left.\dfrac{\dd}{\dd s}\right|_{s=0}
			\left.\dfrac{\dd}{\dd r}\right|_{r=0} \Big(L_{\exp(tu)}^{\rhd}
			L_{\exp(sv)}^{\rhd}\exp(rw)-L_{\exp(sv)}^{\rhd}L_{\exp(tu)}^{\rhd}\exp(rw)\Big)\\
			=&\ \left.\dfrac{\dd}{\dd t}\right|_{t=0} \left.\dfrac{\dd}{\dd s}\right|_{s=0}
			\left.\dfrac{\dd}{\dd r}\right|_{r=0} \Big(L_{\exp(tu)\cdot_{(\mfrakL_{\frac{1}{\infty}},\mfrakH_{\frac{1}{\infty}})}
				\big(\exp(tu)\rhd\exp(sv)\big)}^{\rhd}\exp(rw)-L_{\exp(sv)
				\cdot_{(\mfrakL_{\frac{1}{\infty}},\mfrakH_{\frac{1}{\infty}})}\big(\exp(sv)\rhd\exp(tu)\big)}^{\rhd}\exp(rw)\Big)\\
			=&\ \left.\dfrac{\dd}{\dd t}\right|_{t=0} \left.\dfrac{\dd}{\dd s}\right|_{s=0}
			\left.\dfrac{\dd}{\dd r}\right|_{r=0} L_{\exp(tu)\cdot_{(\mfrakL_{\frac{1}{\infty}},\mfrakH_{\frac{1}{\infty}})}\exp(sv)}^{\rhd}\exp(rw)
			+\left.\dfrac{\dd}{\dd t}\right|_{t=0} \left.\dfrac{\dd}{\dd s}\right|_{s=0}
			\left.\dfrac{\dd}{\dd r}\right|_{r=0}  L_{\exp(tu)\rhd\exp(sv)}^{\rhd}\exp(rw)\\
			&- \left.\dfrac{\dd}{\dd t}\right|_{t=0}
			\left.\dfrac{\dd}{\dd s}\right|_{s=0} \left.\dfrac{\dd}{\dd r}\right|_{r=0}
			L_{\exp(sv)\cdot_{(\mfrakL_{\frac{1}{\infty}},\mfrakH_{\frac{1}{\infty}})}\exp(tu)}^{\rhd}\exp(rw)
			-\left.\dfrac{\dd}{\dd t}\right|_{t=0} \left.\dfrac{\dd}{\dd s}\right|_{s=0}
			\left.\dfrac{\dd}{\dd r}\right|_{r=0}L_{\exp(sv)\rhd\exp(tu)}^{\rhd}\exp(rw)\\
			=&\ \left.\dfrac{\dd}{\dd t}\right|_{t=0} \left.\dfrac{\dd}{\dd s}\right|_{s=0}
			\left.\dfrac{\dd}{\dd r}\right|_{r=0} L_{\exp\Big(\lim\limits_{n\to\infty}\MfrakH_{\frac{1}{n}}
				\big(t\MfrakL_{\frac{1}{n}}(u)+s\MfrakH_{\frac{1}{n}}(v)
				+\frac{1}{2}[t\MfrakL_{\frac{1}{n}}(u),s\mfrakL_{\frac{1}{n}}(v)]
				+\cdots\big)\Big)}^{\rhd}\exp(rw)\\
			&+ \left.\dfrac{\dd}{\dd t}\right|_{t=0} \left.\dfrac{\dd}{\dd s}\right|_{s=0}
			\left.\dfrac{\dd}{\dd r}\right|_{r=0}  L_{\exp(tu)\rhd\exp(sv)}^{\rhd}\exp(rw)\\
			-&\  \left.\dfrac{\dd}{\dd t}\right|_{t=0} \left.\dfrac{\dd}
			{\dd s}\right|_{s=0} \left.\dfrac{\dd}{\dd r}\right|_{r=0}
			L_{\exp\Big(\lim\limits_{n\to\infty}\MfrakH_{\frac{1}{n}}
				\big(s\MfrakH_{\frac{1}{n}}(v)+t\MfrakL_{\frac{1}{n}}(u)
				+\frac{1}{2}[s\MfrakH_{\frac{1}{n}}(v),t\MfrakL_{\frac{1}{n}}(u)]
				+\cdots\big)\Big)}^{\rhd}\exp(rw)\\
			&-  \left.\dfrac{\dd}{\dd t}\right|_{t=0} \left.\dfrac{\dd}
			{\dd s}\right|_{s=0} \left.\dfrac{\dd}{\dd r}\right|_{r=0}
			L_{\exp(sv)\rhd\exp(tu)}^{\rhd}\exp(rw)\\
			=&\  \frac{1}{2}\bigg(\lim_{n\to\infty}\MfrakH_{\frac{1}{n}}
			\Big([\MfrakL_{\frac{1}{n}}(u),\MfrakL_{\frac{1}{n}}(v)]\Big)\bigg)\xsj w
			+(u\xsj v)\xsj w  -\frac{1}{2}\bigg(\lim_{n\to\infty}
			\MfrakH_{\frac{1}{n}}\Big([\MfrakL_{\frac{1}{n}}(v),
			\MfrakL_{\frac{1}{n}}(u)]\Big)\bigg)\xsj w-(v\xsj u)\xsj w\\
			=&\ \bigg(\lim_{n\to\infty}\MfrakH_{\frac{1}{n}}
			\Big([\MfrakL_{\frac{1}{n}}(u),\MfrakL_{\frac{1}{n}}(v)]\Big)\bigg)\xsj w
			+(u\xsj v)\xsj w  -(v\xsj u)\xsj w\quad  u,v,w\in \frakg.
		\end{align*}
	}
Here the Lie bracket is induced from $\cdot$, so Eq.~\meqref{eq:postlimit} holds. Thus $(g,[\cdot,\cdot],\xsj)$ is a post-Lie algebra with limit-weight $\lim\limits_{n\to\infty}(\mfrakL_{\frac{1}{n}},\mfrakH_{\frac{1}{n}})$. Further, if $G$ is a pre-Lie group, applying Lemma~\mref{lemma:prec}, we obtain
\vspace{-.2cm}
\begin{align*}
u \xsj(v\xsj w)-v\xsj(u\xsj w) =&\ \bigg(\lim_{n\to\infty}\MfrakH_{\frac{1}{n}}
\Big([\MfrakL_{\frac{1}{n}}(u),\MfrakL_{\frac{1}{n}}(v)]\Big)\bigg)\xsj w
+(u\xsj v)\xsj w  -(v\xsj u)\xsj w\\
=&\ (u\xsj v)\xsj w-(v\xsj u)\xsj w,\quad  u,v,w\in \frakg.
\vspace{-.2cm}
\end{align*}
Thus $(g,\xsj)$ is a pre-Lie algebra.
\end{proof}

The follow construction is the group version of the fact that a Rota-Baxter Lie algebra with weight zero induces a pre-Lie algebra.

\begin{prop}
	Let $(G,\cdot)$ be an action-\complim $\lim\limits_{n\to\infty}(\mfrakL_{\frac{1}{n}},\mfrakH_{\frac{1}{n}})$-group for an action  $\RE:G\rightarrow \Aut(G)$ and $\frakB$ a relative Rota-Baxter operator with limit-weight $\lim\limits_{n\to\infty}(\mfrakL_{\frac{1}{n}},\mfrakH_{\frac{1}{n}})$.
	Define
	\begin{align*}
		\rhd:G\times G\rightarrow G,\quad
		(a,b)\mapsto \lim\limits_{n\to\infty}\mfrakH_{\frac{1}{n}}\RE_{\frakB(a)}\bigg(\Big(\mfrakL_{\frac{1}{n}}(b)\Big)\bigg)\,.
	\end{align*}
Then $(G,\cdot, \rhd)$ is a post-group of limit-weight $\lim\limits_{n\to\infty}(\mfrakL_{\frac{1}{n}},\mfrakH_{\frac{1}{n}})$. Further if $G$ is a $\lim\limits_{n\to\infty}(\mfrakL_{\frac{1}{n}},\mfrakH_{\frac{1}{n}})$-abelian group, then $(G,\cdot,\rhd)$ is a pre-group.
	\mlabel{prop:preRB}
\end{prop}

\begin{proof}
	Let $a,b,c\in G$. The two defining conditions in Eqs.~(\mref{postgp1}) and~(\mref{postgp2}) follow respectively from
	{\small
		\begin{align*}
			a\rhd(b\cdot_{(\mfrakL_{\frac{1}{\infty}},\mfrakH_{\frac{1}{\infty}})}c)& =\lim_{n\to\infty}
			\mfrakH_{\frac{1}{n}}\Bigg(\RE_{\frakB(a)}\bigg(\mfrakL_{\frac{1}{n}}
			\bigg(\lim_{m\to\infty}\mfrakH_{\frac{1}{m}}
			\Big(\mfrakL_{\frac{1}{m}}(b)
			\mfrakL_{\frac{1}{m}}(c)\Big)\bigg)\bigg)\Bigg)\\
			&= \lim_{n=m\to\infty}\mfrakH_{\frac{1}{n}}
			\Bigg(\RE_{\frakB(a)}\bigg(\mfrakL_{\frac{1}{n}}\bigg(\lim_{m\to\infty}\mfrakH_{\frac{1}{m}}
			\Big(\mfrakL_{\frac{1}{m}}(b)\mfrakL_{\frac{1}{m}}(c))\bigg)\bigg)\Bigg)\ \ \text{(by Eq.~(\mref{eq:jx}))}\\
			&= \lim_{n\to\infty}\mfrakH_{\frac{1}{n}}\bigg(\RE_{\frakB(a)}\Big(\mfrakL_{\frac{1}{n}}(b)
			\mfrakL_{\frac{1}{n}}(c)\Big)\bigg)= \lim_{n\to\infty}\mfrakH_{\frac{1}{n}}\bigg(\RE_{\frakB(a)}\Big(\mfrakL_{\frac{1}{n}}(b)
			\Big)\RE_{\frakB(a)}\Big(\mfrakL_{\frac{1}{n}}(c)
			\Big)\bigg)\\
			&= \Bigg(\lim_{n\to\infty}\mfrakH_{\frac{1}{n}}\bigg(\RE_{\frakB(a)}\Big(\mfrakL_{\frac{1}{n}}(b)
			\Big)\bigg)\Bigg)
			\cdot_{(\mfrakL_{\frac{1}{\infty}},\mfrakH_{\frac{1}{\infty}})}\Bigg(\lim_{n\to\infty}
			\mfrakH_{\frac{1}{n}}\bigg(\RE_{\frakB(a)}\Big(\mfrakL_{\frac{1}{n}}(c)
			\Big)\bigg)\Bigg)\\
			&= (a\rhd b)\cdot_{(\mfrakL_{\frac{1}{\infty}},\mfrakH_{\frac{1}{\infty}})}(a\rhd c)
		\end{align*}
	}
	and
	{\small
		\begin{align*}
			&\Big(a\cdot_{(\mfrakL_{\frac{1}{\infty}},\mfrakH_{\frac{1}{\infty}})}(a\rhd b)\Big)\rhd c
			=\Bigg(\lim_{m\to\infty}\mfrakH_{\frac{1}{m}}\bigg(
			\mfrakL_{\frac{1}{m}}(a)\RE_{\frakB(a)}\Big(\mfrakL_{\frac{1}{n}}(b)
			\Big)\bigg)\Bigg)\rhd c\\
			&= \lim_{n\to\infty}\mfrakH_{\frac{1}{n}}
			\Bigg(\RE_{\frakB\bigg(\lim\limits_{m\to\infty}\mfrakH_{\frac{1}{m}}
				\Big(\mfrakL_{\frac{1}{m}}(a)\RE_{\frakB(a)}\big(\mfrakL_{\frac{1}{n}}(b)
				\big)\Big)\bigg)}\Big(\mfrakL_{\frac{1}{n}}(c)\Big)\Bigg)\\
			&= \lim_{n\to\infty}\mfrakH_{\frac{1}{n}}\Bigg(\RE_{\frakB(a)\frakB(b)}\bigg(
			\mfrakL_{\frac{1}{n}}(c)\bigg)\Bigg)= \lim_{n\to\infty}\mfrakH_{\frac{1}{n}}\Bigg(\RE_{\frakB(a)}\bigg(\RE_{\frakB(b)}\Big(
			\mfrakL_{\frac{1}{n}}(c)\Big)\bigg)\Bigg)\\
			&\hspace{1cm} \text{(by the limit-action propperty and Eq.~(\mref{eq:jx}))}\\
			&= a\rhd(b\rhd c). \qedhere
		\end{align*}
	}
\end{proof}

We now give an example of Rota-Baxter groups with weight zero and hence an example of pre-groups. Note that the group itself is not abelian.

\begin{prop}
	Let $\frakg=\{ (a_{ij}) \in \mathbb{C}^{3 \times 3} \mid a_{ij} = 0\,\text{ for }\,i\geqslant j\}$ be a nilpotent matrix Lie algebra with a Rota-Baxter operator $B$ with weight zero on $\frakg$ satisfying $\frakB([\frakg,\frakg])\subseteq [\frakg,\frakg]$ and $G:=\exp(\frakg)$ be the simply connected nilpotent analytic Lie group whose tangent space is $\frakg$.
\begin{enumerate}
\item \cite[Proposition 3.18]{GGH}	Define a map
	\begin{equation*}
		\frakB:G\rightarrow G, \quad \exp(u)\mapsto \exp\bigg(B(u)+\frac{1}{2}B\Big([u,B(u)]\Big)\bigg).
		\label{eq:exafb}
	\end{equation*}
	Then $(G,\frakB)$ is a Rota-Baxter group with limit-weight zero.
\mlabel{it:rbpl1}
\item
Define a multiplication
\begin{align*}
	\rhd:G\times G\rightarrow G,\quad
	(\exp(u),\exp(v))\mapsto \lim\limits_{n\to\infty}\Big(\frakB(\exp(u))(\exp(v))^{\frac{1}{n}}\frakB(\exp(u))^{-1}\Big)^{n}, \quad u, v\in \frakg.
\end{align*}
Then $(G,\cdot,\rhd)$ is a pre-group, where the multiplication $\cdot$ is the matrix multiplication.
\mlabel{it:rbpl2}
\end{enumerate}
\mlabel{prop:rb33}
\end{prop}

\begin{proof}

\noindent
\meqref{it:rbpl2}
	To apply Proposition ~\mref{prop:preRB}, we verify that its conditions are satisfied.
	First, the existence of  $\lim\limits_{n\to\infty}(\frakB(a)b^{\frac{1}{n}}\frakB(a)^{-1})^{n}$ and the action-\complim property follows from
 {\small
	\begin{align*}
     &\lim\limits_{n\to\infty}\Big(\frakB(\exp(u_n))(\exp(v_n))^{\frac{1}{n}}\frakB(\exp(u_n))^{-1}\Big)^{n}\\
		=&\ \lim_{n\to\infty} \exp \Bigg(n\bigg(\frac{1}{n}v_n+[B(u_n),\frac{1}{n}v_n]\bigg)\Bigg)\hspace{1cm}\text{(by the nilpotent of $\frakg$)}\\
		=&\ \exp\Bigg(v+[B(u),v]\Bigg)\in G.
	\end{align*}}
	Next, $G$ being a $(\pown^{-1},\pown)$-abelian group follows from Example~\mref{coro:0mult} and
	\begin{align*}
		&\lim_{n\to\infty}\Big(\exp(nu)\exp(sv)\exp(-nu)\exp(-nv)\Big)^{n}=\lim_{n\to\infty}\exp\Big(n([\frac{1}{n}u,\frac{1}{n}v])\Big)=\exp(0).
	\end{align*}
	In conclusion, $(G,\cdot,\rhd)$ is a pre-group by Proposition ~\mref{prop:preRB}.
\end{proof}

\begin{remark}
	If the group in Proposition~\mref{prop:rb33} is a Rota-Baxter Lie group with weight zero, then the tangent map of $$\lim\limits_{n\to\infty}\Big(\frakB(\exp(u))\exp(v)^{\frac{1}{n}}\frakB(\exp(u))^{-1}\Big)^{n}$$ is $[\frakB(u),v]$. According to Eq.~(\mref{preLie2}), the Lie bracket defined by
	$$u\xsj v-v\xsj u=[\frakB(u),v]+[u,\frakB(v)]$$
	is exactly the Lie bracket of the descent Lie algebra of the Rota-Baxter Lie algebra with weight zero.
\end{remark}

\subsection{From limit-weighted Rota-Baxter groups to Yang-Baxter equation}
In this subsection, we utilize limit-weighted Rota-Baxter groups to produce set-theoretic solutions of the Yang-Baxter equation. Let us first recall the concept of a skew left brace.

\begin{defn}~\mcite{BG,GV}
	\begin{enumerate}
		\item A {\bf bigroup} $(G,\bullet,\circ)$ is a set $G$ with two binary operations $\bullet$ and $\circ$ such that $(G,\bullet)$ and $(G,\circ)$ are groups.
		
		\item    A bigroup $(G,\bullet, \circ)$ is called a {\bf skew left brace} if
		\begin{equation*}
			a\circ(b\bullet c) = (a\circ b)\bullet a^{-1} \bullet (a\circ c),\quad a,b,c \in G,
		\end{equation*}
		where $a^{-1}$ denotes the inverse of $a$ with respect to $\bullet$.
	\end{enumerate}
\end{defn}

A limited-weighted Rota-Baxter group with bijective pairs induces a skew left brace.

\begin{prop}\cite{BGST}
	Let $(G,\cdot,\rhd)$ be a post group. Then $(G,\ast , \cdot)$ is a skew left brace, where $*$ is defined in Lemma~\mref{lem:prere}.
	\mlabel{Prop: brace}
\end{prop}

Recall that a {\bf set-theoretical solution} to the Yang-Baxter equation (YBE) on a set $X$ is a bijective map $S: X \times X \rightarrow X \times X$ satisfying
\begin{equation*}
	(S\times \id)(\id \times S)(S \times \id) = (\id \times S)(S \times \id)(\id \times S).
\end{equation*}

Set-theoretical solutions of the YBE can be obtained from skew left braces as follows.

\begin{lemma}~\mcite{BG,GV}.
Let $(G,\bullet,\circ)$ be a skew left brace.
\begin{enumerate}
\item There is a group homomorphism
\begin{equation*}
\Omega : (G,\circ)\rightarrow \Aut(G,\bullet),\quad a \mapsto \Omega_{a}
\,\text{ with }\, \Omega_a(b):= a^{-1}\bullet(a\circ b) \,\text{ for }\, b\in G.
\mlabel{eq:psia}
\end{equation*}
Here $a^{-1}$ is the inverse of $a$ for the multiplication $\bullet$.
		\mlabel{it:soltoybe1}
		
\item The map
\begin{equation*}
S : G \times G \rightarrow G \times G ,\quad (a,b) \mapsto \bigg(\Omega_{a}(b), \Omega_{\Omega_{a}(b)}^{-1}\Big((a\circ b)^{-1}\bullet a\bullet(a\circ b)\Big)\bigg)
\end{equation*}
		is a non-degenerate set-theoretical solution to the YBE.
		\mlabel{it:soltoybe2}
	\end{enumerate}
	\mlabel{lem:soltoybe}
\end{lemma}

Notice that $\Omega_{a}^{-1}(b) = a^{\dagger} \circ (a\bullet b)$, where $a^{\dagger}$
is the inverse of $a$ with respect to $\circ$. Limit-weighted Rota-Baxter groups give rise to set-theoretical solutions to the Yang-Baxter equation as follows.

\begin{prop}
	\mlabel{pp:slotoybe2}
	 Let $G$ be an action \complim complete $\lim\limits_{n\to\infty}(\mfrakL_{\frac{1}{n}},\mfrakH_{\frac{1}{n}})$-group for the action  $\RE:G\rightarrow \Aut(G)$  with bijective pairs
	$(\mfrakL_{\frac{1}{n}},\mfrakH_{\frac{1}{n}}), n\in \mathbb{P}$. Let $\frakB$ be a relative Rota-Baxter operator with limit-weight $\lim\limits_{n\to\infty}(\mfrakL_{\frac{1}{n}},\mfrakH_{\frac{1}{n}})$ with respect to $\RE$ on $G$. Then
	\vspace{-.2cm}
	\begin{equation*}
		S: G \times G \rightarrow G\times G ,\quad
		(a,b)\mapsto \bigg(\Omega_{a}(b),\lim_{n\to\infty}\mfrakH_{\frac{1}{n}}\Big(\RE_{\frakB(\Omega_{a}(b))^{-1}}\ad_{\Omega_{a}(b)}\mfrakL_{\frac{1}{n}}(a)\Big)\bigg),
\vspace{-.2cm}
	\end{equation*}
	is a non-degenerate set-theoretical solution to the  Yang-Baxter equation, where
	$$\Omega_{a}(b)=\lim\limits_{n \to \infty }\mfrakH_{\frac{1}{n}}\bigg(\RE_{\frakB(a)}\Big(\mfrakL_{\frac{1}{n}}(b)\Big)\bigg).$$
\end{prop}

\begin{proof}
	By Proposition~\mref{Prop: brace}, $(G,\cdot_{\mfrakL_{\frac{1}{\infty}}},\ast)$ is a skew left brace with
	$$ a \ast b := \lim\limits_{n\to \infty}\mfrakH_{\frac{1}{n}}\bigg(\mfrakL_{\frac{1}{n}}(a)\RE_{\frakB(a)}\Big(\mfrakL_{\frac{1}{n}}(b)\Big)\bigg).$$
	Now
	\begin{align*}
		\Omega_{a}(b)=&\ \lim\limits_{n \to \infty }\mfrakH_{\frac{1}{n}}\bigg(\RE_{\frakB(a)}\Big(\mfrakL_{\frac{1}{n}}(b)\Big)\bigg),\\
		\Omega_{a}^{-1}(b) =&\  a^{\dagger} \ast (a  \cdot_{\mfrakL_{\frac{1}{\infty}}}b)=\lim_{n\to\infty}\mfrakH_{\frac{1}{n}}\Big(\RE_{\frakB(a)^{-1}}\mfrakL_{\frac{1}{n}}(b)\Big),\\
		(a\ast b)^{-1}\cdot_{\mfrakL_{\frac{1}{\infty}}} a\cdot_{\mfrakL_{\frac{1}{\infty}}}(a\ast b)=&\ \lim_{n\to\infty}\mfrakH_{\frac{1}{n}}\Big(\ad_{\frakB(a)\mfrakL(b)\frakB(a)^{-1}}\mfrakL_{\frac{1}{n}}(a)\Big)
		=\lim_{n\to\infty}\mfrakH_{\frac{1}{n}}\Big(\ad_{\Omega_{a}(b)}\mfrakL_{\frac{1}{n}}(a)\Big).
	\end{align*}
	Thus
	\begin{align*}
		\Omega_{\Omega_{a}(b)}^{-1}\Big((a\ast b)^{-1}\cdot_{\mfrakL_{\frac{1}{\infty}}} a\cdot_{\mfrakL_{\frac{1}{\infty}}}(a\ast b)\Big)=&\lim_{n\to\infty}\mfrakH_{\frac{1}{n}}\Bigg(\RE_{\frakB(\Omega_{a}(b))^{-1}}\mfrakL_{\frac{1}{n}}
		\bigg(\mfrakH_{\frac{1}{n}}\Big(\ad_{\Omega_{a}(b)}\mfrakL_{\frac{1}{n}}(a)\Big)\bigg)\Bigg)\\
		=&\ \lim_{n\to\infty}\mfrakH_{\frac{1}{n}}\Big(\RE_{\frakB(\Omega_{a}(b))^{-1}}\ad_{\Omega_{a}(b)}\mfrakL_{\frac{1}{n}}(a)\Big),
	\end{align*}
	and so the result is valid by Lemma~\mref{lem:soltoybe}~(\mref{it:soltoybe2}).
\end{proof}
\vspace{-.3cm}

\section{Relative differential operators and Novikov structures on groups and Lie algebras}
\mlabel{sec:nov}
This section introduces the notions of differential operators on groups and Lie algebras with limit-weights, and of Novikov groups and Novikov Lie algebra. The expected relations among them are established, as summarized in the diagram in Figure~\ref{fig:diffalggp} in page~\pageref{fig:diffalggp}.

We first give the notion of limit-weighted relative differential operator on groups and Lie algebras.

\begin{defn}
\mlabel{defn:reldf}
Let $(H,\cdot_{H})$ be a group and $(G,\cdot_{G})$ be an action-\complim $\lim\limits_{n\to\infty}(\mfrakL_{\frac{1}{n}},\mfrakH_{\frac{1}{n}})$-group for an action $\RE:H\rightarrow \Aut(G)$. A map $\frakD:H\rightarrow G$ is called a {\bf relative differential operator with limit-weight $\lim\limits_{n\to\infty}(\mfrakL_{\frac{1}{n}},\mfrakH_{\frac{1}{n}})$ on $(G,\cdot_{G})$} for the action $\RE$ if
\begin{equation*}
	\mlabel{eq:reldfzero}
	\frakD(a\cdot_{H}b)=\lim_{n\to\infty}
	\mfrakH_{\frac{1}{n}}\bigg(\mfrakL_{\frac{1}{n}}\frakD(a)
	\cdot_{G} \RE_{a} \Big(\mfrakL_{\frac{1}{n}}\frakD(b)\Big) \bigg), \quad  a,b\in H.
\end{equation*}
\end{defn}

\begin{defn}
\mlabel{defn:reliedf}
Let $(\frakh,[\cdot,\cdot]_{\frakh})$ be a Lie algebra and $(\frakg,[\cdot,\cdot]_{\frakg})$ an action-\complim $\lim\limits_{n\to\infty}(\MfrakL_{\frac{1}{n}},\MfrakH_{\frac{1}{n}})$-Lie  algebra for the Lie algebra action $\re:\frakh \rightarrow \Der(\frakg)$. A linear operator $D:\frakh\to \frakg$ is called a {\bf relative differential operator with limit-weight $\lim\limits_{n\to\infty}(\MfrakL_{\frac{1}{n}},\MfrakH_{\frac{1}{n}})$ on $\frakg$ for the action $\re$} if
\begin{equation*}
D([u,v]_{\frakh})
=\lim_{n\to\infty}\MfrakH_{\frac{1}{n}}\bigg(\re_{u}\Big(\MfrakL_{\frac{1}{n}}D(v)\Big)
-\re_{v}\Big(\MfrakL_{\frac{1}{n}}D(u)\Big)+[\MfrakL_{\frac{1}{n}}D(u),
\MfrakL_{\frac{1}{n}}D(v)]_{\frakg}\bigg),\quad   u, v\in \frakh.
    \end{equation*}
\end{defn}
\begin{remark}
    When we take $\MfrakL_{\frac{1}{n}}:=\frac{1}{n} \id_\frakg$ , $\MfrakH_{\frac{1}{n}}:=n\, \id_\frakg$ and adjoint as the action in Definition~\ref{defn:reliedf}, we can recover the notion of differential operator with weight zero.             \mlabel{rk:redf}
\end{remark}

We now show that the tangent map of a limit-weighted relative differential operator on a Lie group is a limit-weighted relative differential operator on the corresponding Lie algebra.

\begin{theorem}
Let $(H,\cdot_{H})$ be a Lie group and Lie group $(G,\cdot_{G})$ an action-\complim $\lim\limits_{n\to\infty}(\mfrakL_{\frac{1}{n}},\mfrakH_{\frac{1}{n}})$-group for the smooth action $\RE:H\rightarrow \Aut(G)$. Let $ \frakD:H\rightarrow G$ be a smooth relative differential operator with limit-weight $\lim\limits_{n\to\infty}(\mfrakL_{\frac{1}{n}},\mfrakH_{\frac{1}{n}})$ for action $\RE$ with $\mfrakL,\mfrakH$ smooth and unital maps.  Let $\frakg = T_e G$ be the Lie algebra of $G$ where the topology on $\frakg$ is induced from $G$, and let
$$
D:=\frakD_{*e},\quad \MfrakL_{\frac{1}{n}}:=(\mfrakL_{\frac{1}{n}})_{*e},\quad \MfrakH_{\frac{1}{n}}:=(\mfrakH_{\frac{1}{n}})_{*e}
$$
be the tangent maps at $e$. Let $\frakh=T_eH$ be the Lie algebra of $H$ with the topology on $\frakh$ induced from $H$. Then $D:\frakh\rightarrow\frakg$ is a relative differential operator of limit-weight $\lim\limits_{n\to\infty}(\MfrakL_{\frac{1}{n}},\MfrakH_{\frac{1}{n}})$ for the action $\re:=\RE_{*}:\frakh\to \Der(\frakg)$.
\mlabel{thm:dgpdla}
\end{theorem}
\begin{proof}
For $u, v \in \frakh$, we have
\begin{align*}
 D([ u, v  ]_{\frakh})
=&\  \frac{\mathrm{d}^2}{\mathrm{d}t \mathrm{d}s}\bigg|_{t,s=0} \frakD(e^{tu}\cdot_H e^{sv}\cdot_H e^{-tu})\\
=&\  \frac{\mathrm{d}^2}{\mathrm{d}t \mathrm{d}s}\bigg|_{t,s=0} \lim_{n\to\infty}
	\mfrakH_{\frac{1}{n}}\bigg(\mfrakL_{\frac{1}{n}}\frakD(e^{tu}\cdot_H e^{sv})
	\cdot_{G} \RE_{e^{tu}\cdot_H e^{sv}} \Big(\mfrakL_{\frac{1}{n}}\frakD(e^{-tu})\Big) \bigg)\\
=&\  \frac{\mathrm{d}^2}{\mathrm{d}t \mathrm{d}s}\bigg|_{t,s=0} \lim_{n\to\infty}
	\mfrakH_{\frac{1}{n}}\bigg(\mfrakL_{\frac{1}{n}}\frakD(e^{tu})
	\cdot_{G} \RE_{e^{tu}} \Big(\mfrakL_{\frac{1}{n}}\frakD(e^{sv})\Big)
	\cdot_{G} \RE_{e^{tu}\cdot_H e^{sv}} \Big(\mfrakL_{\frac{1}{n}}\frakD(e^{-tu})\Big) \bigg)\\
=&\ \lim_{n\to\infty} \frac{\mathrm{d}^2}{\mathrm{d}t \mathrm{d}s}\bigg|_{t,s=0}
	\mfrakH_{\frac{1}{n}}\bigg(\mfrakL_{\frac{1}{n}}\frakD(e^{tu})
	\cdot_{G} \RE_{e^{tu}} \Big(\mfrakL_{\frac{1}{n}}\frakD(e^{sv})\Big)
	\cdot_{G} \RE_{e^{tu}\cdot_H e^{sv}} \Big(\mfrakL_{\frac{1}{n}}\frakD(e^{-tu})\Big) \bigg)\\
=&\ \lim_{n\to\infty}\MfrakH_{\frac{1}{n}}\bigg(\re_{u}\Big(\MfrakL_{\frac{1}{n}}D(v)\Big)-\re_{v}\Big(\MfrakL_{\frac{1}{n}}D(u)\Big)+[\MfrakL_{\frac{1}{n}}D(u),\MfrakL_{\frac{1}{n}}D(v)]_{\frakg}\bigg).
\end{align*}
Then the result follows from Proposition~\mref{prop:leiswl}.
\end{proof}

We next give the notion of a Novikov group.
\begin{defn}
Let $G$ be an action-\complim $\lim\limits_{n\to\infty}(\mfrakL_{\frac{1}{n}},\mfrakH_{\frac{1}{n}})$-group for an action $\RE:G\rightarrow \Aut(G)$ and also a $\lim\limits_{n\to\infty}(\mfrakL_{\frac{1}{n}},\mfrakH_{\frac{1}{n}})$-abelian group. Let $\No$ be a binary operation on $G$. We call $(G,\No)$ a {\bf Novikov group with limit-weight $\lim\limits_{n\to\infty}(\mfrakL_{\frac{1}{n}},\mfrakH_{\frac{1}{n}})$ for the action $\RE$} if for each $a\in G$, we have $a\No e=e$ and for each $a,b,c\in G$, there are
\begin{equation}
a\No(bc)=(a\No b)\cdot_{(\mfrakL_{\frac{1}{\infty}},\mfrakH_{\frac{1}{\infty}})}\Big((ab)\No c\Big)
\mlabel{eq:alesbc}
\end{equation}
and
\begin{equation*}
\mlabel{eq:rets}
\lim\limits_{n\to\infty}\mfrakH_{\frac{1}{n}}\RE_a\Big(\mfrakL_{\frac{1}{n}}(b\No c)\Big)=(ab\No c).
\end{equation*}
If $G$ is in addition a $\mathbb{Q}$-group and also an action-\complim $\lim\limits_{n\to\infty}(\pown^{-1},\pown)$-group for the adjoint action, with $\pown$ defined in Eq.~\meqref{eq:powern}. Then we call $(G,\No)$ a {\bf Novikov group}.
\end{defn}

In analog to the Gelfand-Dorfman construction~\mcite{GD} of Novikov algebras from differential operators on commutative algebras, we have the following relation.

\begin{prop}
\mlabel{pp:novidf}
Let $(G,\cdot)$ be an action-\complim $\lim\limits_{n\to\infty}(\mfrakL_{\frac{1}{n}},\mfrakH_{\frac{1}{n}})$-group for an action $\RE:G\rightarrow \Aut(G)$.
\begin{enumerate}
\item Let $\frakD:G\rightarrow G$ be a relative differential operator with limit-weight $\lim\limits_{n\to\infty}(\mfrakL_{\frac{1}{n}},\mfrakH_{\frac{1}{n}})$ on $(G,\cdot)$ for the action $\RE$. Define a multiplication $\No$ on $G$ by
\begin{equation}
\No: G \times G \rightarrow G,\quad (a,b)\mapsto a\No b:=\lim_{n\to\infty}\mfrakH_{\frac{1}{n}}\RE_{a} \Big(\mfrakL_{\frac{1}{n}}\frakD(b)\Big).
\mlabel{eq:ggles}
\end{equation}
Then $(G,\No)$ is a Novikov group with limit-weight $\lim\limits_{n\to\infty}(\mfrakL_{\frac{1}{n}},\mfrakH_{\frac{1}{n}})$ for the action $\RE$.
\mlabel{it:novidfa}
\item \mlabel{it:novidfb}
Let $(G,\No)$ be a Novikov group with limit-weight $\lim\limits_{n\to\infty}(\mfrakL_{\frac{1}{n}},\mfrakH_{\frac{1}{n}})$ for the action $\RE$. Then the map
$$\frakD:G\rightarrow G, \quad a\mapsto e\No a, $$
is a relative differential operator with limit-weight $\lim\limits_{n\to\infty}(\mfrakL_{\frac{1}{n}},\mfrakH_{\frac{1}{n}})$ on $(G,\cdot)$ for the action $\RE$.
\end{enumerate}
\end{prop}

\begin{proof}  Let $a,b,c\in G$.
\meqref{it:novidfa} It follows from
{\small
        \begin{align*}
        a\No (bc)=&\ \lim_{n\to\infty}\mfrakH_{\frac{1}{n}}\RE_{a} \Big(\mfrakL_{\frac{1}{n}}\frakD(bc)\Big)\\
        =&\ \lim_{n\to\infty}\mfrakH_{\frac{1}{n}}\RE_{a} \Bigg(\mfrakL_{\frac{1}{n}}\Bigg(\lim_{m\to\infty}
	\mfrakH_{\frac{1}{m}}\bigg(\mfrakL_{\frac{1}{m}}\frakD(b)
	   \RE_{b} \Big(\mfrakL_{\frac{1}{m}}\frakD(c)\Big) \bigg)\Bigg)\Bigg)\\
        =&\ \lim_{n\to\infty}\mfrakH_{\frac{1}{n}}\RE_{a} \Bigg(\bigg(\mfrakL_{\frac{1}{m}}\frakD(b)
	     \RE_{b} \Big(\mfrakL_{\frac{1}{m}}\frakD(c)\Big) \bigg)\Bigg)\\
        =&\ \lim_{n\to\infty}\mfrakH_{\frac{1}{n}}\Bigg(\RE_{a} \Big(\mfrakL_{\frac{1}{m}}\frakD(b)
	     \Big)\RE_{a}\bigg( \RE_{b} \Big(\mfrakL_{\frac{1}{m}}\frakD(c)\Big)\bigg)\Bigg)\\
        =&\ \lim_{n\to\infty}\mfrakH_{\frac{1}{n}}\bigg(\RE_{a} \Big(\mfrakL_{\frac{1}{m}}\frakD(b)
	     \Big) \RE_{ab} \Big(\mfrakL_{\frac{1}{m}}\frakD(c)\Big)\bigg)\\
        =&\ \Bigg(\lim\limits_{n\to\infty}\mfrakH_{\frac{1}{n}}\bigg(\RE_{a} \Big(\mfrakL_{\frac{1}{m}}\frakD(b)
	     \Big)\bigg)\Bigg) \cdot_{(\mfrakL_{\frac{1}{\infty}},\mfrakH_{\frac{1}{\infty}})} \Bigg(\lim\limits_{n\to\infty}\mfrakH_{\frac{1}{n}}\bigg(\RE_{ab} \Big(\mfrakL_{\frac{1}{m}}\frakD(c)\Big)\bigg)\Bigg)\\
        =&\ (a\No b)\cdot_{(\mfrakL_{\frac{1}{\infty}},\mfrakH_{\frac{1}{\infty}})}\Big((ab)\No c\Big)
    \end{align*}
}
and
    \begin{align*}
        \lim\limits_{n\to\infty}\mfrakH_{\frac{1}{n}}\RE_a\Big(\mfrakL_{\frac{1}{n}}(b\No c)\Big)=&
        \lim_{n\to\infty}\mfrakH_{\frac{1}{n}}\RE_{a} \Bigg(\mfrakL_{\frac{1}{n}}\Bigg(\lim_{m\to\infty}
	\mfrakH_{\frac{1}{m}}\bigg(\mfrakL_{\frac{1}{m}}\frakD(b)
	   \RE_{b} \Big(\mfrakL_{\frac{1}{m}}\frakD(c)\Big) \bigg)\Bigg)\Bigg)\\
=&\lim_{n\to\infty}\mfrakH_{\frac{1}{n}}\RE_{a} \Big(\mfrakL_{\frac{1}{n}}\frakD(bc)\Big)=
        (ab\No c).
    \end{align*}

\noindent
\meqref{it:novidfb}
It follows from
    \begin{align*}
        \frakD(ab)=&\ e\No(ab)\\
        =&\ (e\No a)\cdot_{(\mfrakL_{\frac{1}{\infty}},\mfrakH_{\frac{1}{\infty}})}\Big(a\No b\Big) \quad\quad \text{(by Eq.~(\mref{eq:alesbc}))}\\
        =&\ (e\No a)\cdot_{(\mfrakL_{\frac{1}{\infty}},\mfrakH_{\frac{1}{\infty}})}\bigg(\lim\limits_{n\to\infty}\mfrakH_{\frac{1}{n}}\RE_a\Big(\mfrakL_{\frac{1}{n}}(e\No b)\Big)\bigg) \quad\quad \text{(by Eq.~(\mref{eq:ggles}))}\\
        =&\ \lim_{n\to\infty}
	    \mfrakH_{\frac{1}{n}}\bigg(\mfrakL_{\frac{1}{n}}\frakD(a)
	      \RE_{a} \Big(\mfrakL_{\frac{1}{n}}\frakD(b)\Big) \bigg).  \qedhere
    \end{align*}
\end{proof}

As an application, we give an example of a Novikov group.

\begin{prop}
	\mlabel{noviex}
Let $\frakg=\{ (a_{ij}) \in \mathfrak{gl}_3(\RR)\mid a_{ij} = 0\, \text{ for }\,i\geqslant j\}$ be the nilpotent Lie algebra with a derivation $D$, and $G=\exp(\frakg)=I+\frakg$  be the simply connected nilpotent analytic Lie group $($the first Heisenberg group $H_1(\RR)$$)$ with the action
	\begin{align*}
		\RE:\ G \rightarrow \Aut(G),\quad  \exp(u)\mapsto \exp(u)\exp(v)\exp(-u).
	\end{align*}
	For each $\exp(u),\exp(v)\in G$, define
	$$\exp(u)\No \exp(v) :=\exp\Big(D(v)+\frac{1}{2}[v,D(v)]+\frac{1}{2}[u,D(v)]\Big).$$
	Then for $\mfrakH_{\frac{1}{n}}:=\pown$ defined in Eq.~\meqref{eq:powern}, the pair $(G,\No)$ is a Novikov group with limit-weight $\lim\limits_{n\to\infty}(\pown^{-1},\pown)$ for the action $\RE$.
\end{prop}
\begin{proof}
	It follows from ~\cite[Proposition~3.26]{GGH} that the map $$\frakD:G\to G, \quad \exp(u)\mapsto \exp(D(u)+[u,\frac{1}{2}D(u)])$$
	is a relative differential operator with limit-weight zero. Notice that $$\lim\limits_{n\to\infty}\Bigg(\exp(u)\exp\bigg(\frac{1}{n}\Big(D(v)+\frac{1}{2}[v,D(v)]\Big)\bigg)\exp(-u)\Bigg)^{n}=\exp\Big(D(v)+\frac{1}{2}[v,D(v)]+\frac{1}{2}[u,D(v)]\Big).$$ Thus by Propsition~\mref{pp:novidf}, $(G,\No)$ is a Novikov group with limit-weight $\lim\limits_{n\to\infty}(\pown^{-1},\pown)$ for the action $\RE$.
\end{proof}

Now we give the notion of a Novikov Lie algebra.
Note that a Lie algebra $(\frakg, [\cdot, \cdot])$ with a derivation $D$ induces a magmatic algebra $(\frakg, \no)$ by setting $u\no v:= [u, D(v)]$~\mcite{KSO}.
Here we use a different approach to give another derived structure, named the Novikov Lie algebra $(\frakg, [\cdot, \cdot], \no)$. The key distinction is that the new structure keeps the Lie bracket $[\cdot, \cdot]$ to insure that a Novikov Lie algebra serves as a tangent space of a Novikov group (Theorem~\mref{thm:ngpnla}).

\begin{defn}
Let $(\frakg, [\cdot, \cdot])$ be an action-\complim $\lim\limits_{n\to\infty}(\MfrakL_{\frac{1}{n}},\MfrakH_{\frac{1}{n}})$-Lie algebra for an action $\re:\frakg\rightarrow \Der(\frakg)$ and also a  $\lim\limits_{n\to\infty}(\MfrakL_{\frac{1}{n}},\MfrakH_{\frac{1}{n}})$-abelian Lie algebra. Let $\no$ be a bilinear map on $\frakg$. We call $(\frakg,\no)$  a {\bf Novikov Lie algebra with limit-weight $\lim\limits_{n\to\infty}(\MfrakL_{\frac{1}{n}},\MfrakH_{\frac{1}{n}})$ for the action $\re$} if
\begin{align*}
 u\no ([v,w])=&\lim\limits_{n\to\infty}\MfrakH_{\frac{1}{n}}\bigg(\re_u\Big(\MfrakL_{\frac{1}{n}}(v\no w)\Big)\bigg)-\lim\limits_{n\to\infty}\MfrakH_{\frac{1}{n}}\bigg(\re_u\Big(\MfrakL_{\frac{1}{n}}(w\no v)\Big)\bigg),\\
([u,v])\no w=& \lim\limits_{n\to\infty}\MfrakH_{\frac{1}{n}}\bigg(\re_u\Big(\MfrakL_{\frac{1}{n}}(v\no w)\Big)\bigg)-\lim\limits_{n\to\infty}\MfrakH_{\frac{1}{n}}\bigg(\re_v\Big(\MfrakL_{\frac{1}{n}}\Big((u\no w)\Big)\bigg), \quad u, v, w\in \frakg.
\mlabel{novilie}
\end{align*}
A {\bf Novikov Lie algebra} is a Lie algebra $\frakg$ with a binary operation $\no$ on $\frakg$ satisfying
\begin{align*}
u\no ([v,w])=&[u,(v\no w)]-[u,(w\no v)],\\
[u,v]\no w=&[u,(v\no w)]-[v,(u\no w)],\quad u, v, w \in \frakg.
\end{align*}
\end{defn}

For the relation between these two notions of Novikov Lie algebras we have

\begin{prop}
Let $(\frakg,\no)$ be a Novikov Lie algebra with limit-weight $\lim\limits_{n\to\infty}(\MfrakL_{\frac{1}{n}},\MfrakH_{\frac{1}{n}})$ for an action $\re$. If $\MfrakL_{\frac{1}{n}}=\frac{1}{n}\,\id_\frakg,\MfrakH_{\frac{1}{n}}=n \,\id_\frakg$ and $\gamma$ is the adjoint action, then $\frakg$ is a Novikov Lie algebra.
\mlabel{pp:dervex}
\end{prop}

\begin{proof}
This follows directly from
\begin{align*}
    u\no ([v,w])=&\ \lim\limits_{n\to\infty}\MfrakH_{\frac{1}{n}}\bigg(\re_u\Big(\MfrakL_{\frac{1}{n}}(v\no w)\Big)\bigg)-\lim\limits_{n\to\infty}\MfrakH_{\frac{1}{n}}\bigg(\re_u\Big(\MfrakL_{\frac{1}{n}}(w\no v)\Big)\bigg)\\
        =&\ \lim\limits_{n\to\infty}\Big(n [u,\frac{1}{n}v\no w]\Big)-\lim\limits_{n\to\infty}\Big(n [u,\frac{1}{n}w\no v]\Big)\\
   =&\ [u,(v\no w)]-[u,(w\no v)].
\end{align*}
Similarly, we have $[u,v]\no w=[u,(v\no w)]-[v,(u\no w)]$.
\end{proof}

We next show that a derivation on a Lie algebra gives rise to a Novikov Lie algebra, starting with the more general relation in the relative context.

\begin{prop}
\mlabel{pp:dernov}
Let $(\frakg,[\cdot,\cdot]_{\frakg})$ be an action-\complim $\lim\limits_{n\to\infty}(\MfrakL_{\frac{1}{n}},\MfrakH_{\frac{1}{n}})$-abelian Lie algebra for an action $\re:\frakg \rightarrow \Der(\frakg)$ and $D:\frakg\to \frakg$ be a  relative differential operator with limit-weight $\lim\limits_{n\to\infty}(\MfrakL_{\frac{1}{n}},\MfrakH_{\frac{1}{n}})$ on $\frakg$ for the action $\re$. Define the multiplication $\no$
    \begin{equation}
   \no: \frakg \times \frakg \rightarrow \frakg,\quad (u,v)\mapsto u\no v:=\lim_{n\to\infty}\MfrakH_{\frac{1}{n}}\bigg(\re_{u}\Big(\MfrakL_{\frac{1}{n}}D(v)\Big)\bigg).
    \end{equation}
Then $(\frakg,\no)$ is a Novikov Lie algebra with limit-weight $\lim\limits_{n\to\infty}(\MfrakL_{\frac{1}{n}},\MfrakH_{\frac{1}{n}})$ for the action $\re$.

In particular, if $\re$ is the adjoint action and $\MfrakL_{\frac{1}{n}}=\frac{1}{n}\,\id_\frakg$, $\MfrakH_{\frac{1}{n}}=n\,\id_\frakg$, then $(\frakg,\no)$ is a Novikov Lie algebra.
\end{prop}

\begin{proof}
    For $u, v,w\in \frakg$, we have
    \begin{align*}
        &u\no ([v,w])\\
        =&\ \lim\limits_{n\to\infty}\MfrakH_{\frac{1}{n}}\re_u\Big(\MfrakL_{\frac{1}{n}}D([v,w])\Big)\\
        =&\ \lim\limits_{n\to\infty}\MfrakH_{\frac{1}{n}}\re_u\bigg(\MfrakL_{\frac{1}{m}}\Big(\lim\limits_{m\to\infty}\MfrakH_{\frac{1}{m}}[\MfrakL_{\frac{1}{m}}D(v),w]\Big)\bigg)+\lim\limits_{n\to\infty}\MfrakH_{\frac{1}{n}}\re_u\bigg(\MfrakL_{\frac{1}{n}}\Big(\lim\limits_{m\to\infty}\MfrakH_{\frac{1}{m}}[v,\MfrakL_{\frac{1}{m}}D(w)]\Big)\bigg)\\
        =&\ \lim\limits_{n\to\infty}\MfrakH_{\frac{1}{n}}\bigg(\re_u\Big(\MfrakL_{\frac{1}{n}}(v\no w)\Big)\bigg)-\lim\limits_{n\to\infty}\MfrakH_{\frac{1}{n}}\bigg(\re_u\Big(\MfrakL_{\frac{1}{n}}(w\no v)\Big)\bigg),
    \end{align*}
    and
    \begin{align*}
        [u,v]\no w=&\lim\limits_{n\to\infty}\MfrakH_{\frac{1}{n}}\re_{[u,v]}\Big(\MfrakL_{\frac{1}{n}}D(w)\Big)\\
        =&\lim\limits_{n\to\infty}\MfrakH_{\frac{1}{n}}\re_{u}\Big(\re_v(\MfrakL_{\frac{1}{n}}D(w))\Big)-\lim\limits_{n\to\infty}\MfrakH_{\frac{1}{n}}\re_{v}\Big(\re_u(\MfrakL_{\frac{1}{n}}D(w))\Big) \\
        =&\lim\limits_{n\to\infty}\MfrakH_{\frac{1}{n}}\re_{u}\Bigg(\MfrakL_{\frac{1}{n}}\bigg(\lim\limits_{m\to\infty}\MfrakH_{\frac{1}{m}}\Big(\re_v(\MfrakL_{\frac{1}{m}}D(w))\bigg)\Bigg)-\lim\limits_{n\to\infty}\MfrakH_{\frac{1}{n}}\re_{v}\Bigg(\MfrakL_{\frac{1}{n}}\bigg(\lim\limits_{m\to\infty}\MfrakH_{\frac{1}{m}}\Big(\re_u(\MfrakL_{\frac{1}{m}}D(w))\bigg)\Bigg)\\
        =& \lim\limits_{n\to\infty}\MfrakH_{\frac{1}{n}}\bigg(\re_u\Big(\MfrakL_{\frac{1}{n}}(v\no w)\Big)\bigg)-\lim\limits_{n\to\infty}\MfrakH_{\frac{1}{n}}\bigg(\re_v\Big(\MfrakL_{\frac{1}{n}}\Big((u\no w)\Big)\bigg), \quad u, v, w\in \frakg,
    \end{align*}
as desired.

If $\re$ is the adjoint action and $\MfrakL_{\frac{1}{n}}=\frac{1}{n}\,\id_\frakg$, $\MfrakH_{\frac{1}{n}}=n\,\id_\frakg$, then by Proposition~\mref{pp:dervex}, $(\frakg,\no)$ is a Novikov Lie algebra.
\end{proof}

The tangent space of a Novikov group is a Novikov Lie algebra as follows.

\begin{theorem}
Let $G$ be a Lie group with a smooth multiplication $\No$ such that $(G, \No)$ is a Novikov group with limit-weight $\lim\limits_{n\to\infty}(\mfrakL_{\frac{1}{n}},\mfrakH_{\frac{1}{n}})$ for an action $\RE:G\rightarrow \Aut(G)$, and let $\mfrakL_{\frac{1}{n}},\mfrakH_{\frac{1}{n}},\RE$ be smooth maps and $(\mfrakL_{\frac{1}{n}},\mfrakH_{\frac{1}{n}})$ unital pairs. Let $\frakg$ be the Lie algebra of $G$ where the topology on $\frakg$ is induced from $G$ and $\re:= \RE_{*e}:\frakg \rightarrow \Der(\frakg), \MfrakL_{\frac{1}{n}}:=(\mfrakL_{\frac{1}{n}})_{*e}$, $\MfrakH_{\frac{1}{n}}:=(\mfrakH_{\frac{1}{n}})_{*e}$ and $\no:=\No_{*e}$. Then $\frakg$ is a Novikov Lie algebra for the action $\re$ with limit-weight $\lim\limits_{n\to\infty}(\MfrakL_{\frac{1}{n}},\MfrakH_{\frac{1}{n}})$. In particular, if $(G,\No)$ is a Novikov group, then $(\frakg,\no)$ is a Novikov Lie algebra.
\mlabel{thm:ngpnla}
\end{theorem}

\begin{proof}
    For any $u, v\in \frakg$,  we have
    \begin{align*}
        0=\frac{d}{dt}\Big|_{t=0}\exp(0)=\frac{d}{dt}\Big|_{t,s=0}\exp(tu)\No\exp(sv).
    \end{align*}
    Then for $u, v, w\in \frakg$, we obtain
{\small
\begin{align*}
        &\ u\no ([v,w])\\
        =&\ \left.\dfrac{\dd}{\dd t}\right|_{t=0} \left.\dfrac{\dd}{\dd s}\right|_{s=0}
			\left.\dfrac{\dd}{\dd r}\right|_{r=0} \bigg(\exp(tu)\No \Big(\exp(sv)\exp(rw)\Big)-\exp(tu)\No \Big(\exp(rw)\exp(sv)\Big)\bigg)\\
   =&\ \left.\dfrac{\dd}{\dd t}\right|_{t=0} \left.\dfrac{\dd}{\dd s}\right|_{s=0}
			\left.\dfrac{\dd}{\dd r}\right|_{r=0} \Bigg(\Big(\exp(tu)\No \exp(sv)\Big)\cdot_{(\mfrakL_{\frac{1}{\infty}},\mfrakH_{\frac{1}{\infty}})}\bigg(\Big(\exp(tu)\exp(sv)\Big)\No \exp(rw)\bigg)\Bigg)\\
   &- \left.\dfrac{\dd}{\dd t}\right|_{t=0} \left.\dfrac{\dd}{\dd s}\right|_{s=0}
			\left.\dfrac{\dd}{\dd r}\right|_{r=0} \Bigg(\Big(\exp(tu)\No \exp(rw)\Big)\cdot_{(\mfrakL_{\frac{1}{\infty}},\mfrakH_{\frac{1}{\infty}})}\bigg(\Big(\exp(tu)\exp(rw)\Big)\No \exp(sv)\bigg)\Bigg)\\
   =&\ \lim_{n\to\infty}\Bigg(\left.\dfrac{\dd}{\dd t}\right|_{t=0} \left.\dfrac{\dd}{\dd s}\right|_{s=0}
			\left.\dfrac{\dd}{\dd r}\right|_{r=0}\mfrakH_{\frac{1}{n}}\bigg(\mfrakL_{\frac{1}{n}}\Big(\exp(tu)\No\exp(sv)\Big)\mfrakL_{\frac{1}{n}}\Big(\exp(rw)\Big)\bigg)\\
   &+\left.\dfrac{\dd}{\dd t}\right|_{t=0} \left.\dfrac{\dd}{\dd s}\right|_{s=0}
			\left.\dfrac{\dd}{\dd r}\right|_{r=0}\mfrakH_{\frac{1}{n}}\bigg(\mfrakL_{\frac{1}{n}}\Big(\exp(sv)\Big)\mfrakL_{\frac{1}{n}}\Big(\exp(tu)\No\exp(rw)\Big)\bigg)\\
   &+ \left.\dfrac{\dd}{\dd t}\right|_{t=0} \left.\dfrac{\dd}{\dd s}\right|_{s=0}
			\left.\dfrac{\dd}{\dd r}\right|_{r=0}\bigg(\Big(\exp(tu)\exp(sv)\Big)\No\exp(rw)\bigg)\\
   &-\left.\dfrac{\dd}{\dd t}\right|_{t=0} \left.\dfrac{\dd}{\dd s}\right|_{s=0}
			\left.\dfrac{\dd}{\dd r}\right|_{r=0}\mfrakH_{\frac{1}{n}}\bigg(\mfrakL_{\frac{1}{n}}\Big(\exp(tu)\No\exp(rw)\Big)\mfrakL_{\frac{1}{n}}\Big(\exp(sv)\Big)\bigg)\\
   &+ \left.\dfrac{\dd}{\dd t}\right|_{t=0} \left.\dfrac{\dd}{\dd s}\right|_{s=0}
			\left.\dfrac{\dd}{\dd r}\right|_{r=0}\mfrakH_{\frac{1}{n}}\bigg(\mfrakL_{\frac{1}{n}}\Big(\exp(rw)\Big)\mfrakL_{\frac{1}{n}}\Big(\exp(tu)\No\exp(sv)\Big)\bigg)\\
   &-\left.\dfrac{\dd}{\dd t}\right|_{t=0} \left.\dfrac{\dd}{\dd s}\right|_{s=0}
			\left.\dfrac{\dd}{\dd r}\right|_{r=0}\bigg(\Big(\exp(tu)\exp(rw)\Big)\No\exp(sv)\bigg)\Bigg)\\
   =&\ \lim_{n\to\infty}\Bigg(\left.\dfrac{\dd}{\dd t}\right|_{t=0} \left.\dfrac{\dd}{\dd s}\right|_{s=0}
			\left.\dfrac{\dd}{\dd r}\right|_{r=0}\bigg(\mfrakH_{\frac{1}{n}}\Big(\RE_{\exp(tu)}\mfrakH_{\frac{1}{n}}(\exp(sv)\No\exp(rw))\Big)\\
&\hspace{4cm}-\Big(\RE_{\exp(tu)}\mfrakH_{\frac{1}{n}}(\exp(rw)\No\exp(sv))\Big)\bigg)\Bigg)\\
   &\ \hspace{5cm}\text{(by limit-abelian property and Eq.~(\mref{eq:rets}))}\\
   =&\  \lim\limits_{n\to\infty}\MfrakH_{\frac{1}{n}}\bigg(\re_u\Big(\MfrakL_{\frac{1}{n}}(v\no w)\Big)\bigg)-\lim\limits_{n\to\infty}\MfrakH_{\frac{1}{n}}\bigg(\re_u\Big(\MfrakL_{\frac{1}{n}}(w\no v)\Big)\bigg),
    \end{align*}
}
and
{\small
\begin{align*}
&\ ([u,v])\no w\\
=&\ \left.\dfrac{\dd}{\dd t}\right|_{t=0} \left.\dfrac{\dd}{\dd s}\right|_{s=0}
\left.\dfrac{\dd}{\dd r}\right|_{r=0} \bigg(\Big(\exp(tu)\exp(sv)\Big)\No\exp(rw) -\Big(\exp(sv)\exp(tu)\Big)\No\exp(rw)\bigg)\\
=&\ \left.\dfrac{\dd}{\dd t}\right|_{t=0} \left.\dfrac{\dd}{\dd s}\right|_{s=0}
\left.\dfrac{\dd}{\dd r}\right|_{r=0} \lim\limits_{n\to\infty}\mfrakH_{\frac{1}{n}}\RE_{\exp(tu)}\bigg(\mfrakL_{\frac{1}{n}}\Big(\exp(sv)\No \exp(rw)\Big)\bigg)\\
-&\left.\dfrac{\dd}{\dd t}\right|_{t=0} \left.\dfrac{\dd}{\dd s}\right|_{s=0}
\left.\dfrac{\dd}{\dd r}\right|_{r=0}\lim\limits_{n\to\infty}\mfrakH_{\frac{1}{n}}\RE_{\exp(sv)}\bigg(\mfrakL_{\frac{1}{n}}\Big(\exp(tu)\No \exp(rw)\Big)\bigg)\\
=&\lim\limits_{n\to\infty}\MfrakH_{\frac{1}{n}}\bigg(\re_u\Big(\MfrakL_{\frac{1}{n}}(v\no w)\Big)\bigg)-\lim\limits_{n\to\infty}\MfrakH_{\frac{1}{n}}\bigg(\re_v\Big(\MfrakL_{\frac{1}{n}}\Big((u\no w)\Big)\bigg).
\end{align*}
}
In particular, if $(G,\No)$ is a Novikov group, then we have $\MfrakL_{\frac{1}{n}}=\frac{1}{n}\,\id_\frakg$, $\MfrakH_{\frac{1}{n}}=n\,\id_\frakg$, and $\re$ is the adjoint action. Thus $(\frakg,\no)$ is a Novikov Lie algebra.
\end{proof}

We end the paper by revisiting Proposition~\mref{noviex} and giving an example of Novikov Lie algebras. Applying Proposition~\mref{pp:dernov} and Theorem~\mref{thm:ngpnla} we obtain

\begin{prop}
Let $\frakg=\{ (a_{ij}) \in \mathfrak{gl}_3(\RR)\mid a_{ij} = 0\, \text{ for }\,i\geqslant j\}$ be the nilpotent Lie algebra with a derivation $D$. Then the multiplication $\no$ on $\frakg$ given by
$$ u \no v:=[u,D(v)], \quad u, v\in \frakg,$$
defines a Novikov Lie algebra on $\frakg$. Furthermore, this Novikov Lie algebra is the tangent space of the Novikov group $(G,\No)$ defined in Proposition~\mref{noviex}.
\mlabel{noviex2}
\end{prop}

\noindent
{\bf Acknowledgments.} This work is supported by the National Natural Science Foundation of China (12071191, 12101316) and Innovative Fundamental Research Group Project of Gansu Province (23JRRA684). The authors thank Chengming Bai and Yong Liu for helpful discussions, and Yifei Li for valuable programming.

\noindent
{\bf Declaration of interests.} The authors have no conflicts of interest to disclose.

\noindent
{\bf Data availability.} Data sharing is not applicable as no new data were created or analyzed.

\end{document}